\documentclass{amsart}
\usepackage[utf8x]{inputenc}
\usepackage[T1]{fontenc}
\usepackage{ae,aecompl}

\usepackage{amssymb}
\usepackage{amsmath}

\usepackage{tikz}
\usetikzlibrary{matrix,shapes}

\usepackage{enumerate}
\usepackage{verbatim}
\usepackage{hyperref}

\usepackage{graphicx}

\usepackage{ifpdf}
\ifpdf
\fi

\newenvironment{sageexample}%
{\verbatim\small}%
{\endverbatim}

\newtheorem{theorem}{Theorem}[section]
\newtheorem{lemma}[theorem]{Lemma}
\newtheorem{proposition}[theorem]{Proposition}
\newtheorem{corollary}[theorem]{Corollary}
\newtheorem{definition}[theorem]{Definition}

\newtheorem{problem}[theorem]{Problem}
\newtheorem{exercise}[theorem]{Exercise}
\newtheorem{conjecture}[theorem]{Conjecture}

\newtheorem{remark}[theorem]{Remark}

\theoremstyle{definition}
\newtheorem{example}[theorem]{Example}

\numberwithin{equation}{section}

\newcommand{\card}{\operatorname{Card}}
\newcommand{\cont}{\operatorname{cont}}

\newcommand{\joins}{\operatorname{Joins}}

\newcommand{\id}{\operatorname{id}}
\newcommand{\im}{\operatorname{im}}
\newcommand{\longest}{\pi}
\newcommand{\K}{\mathbb{K}}
\newcommand{\ZZ}{\mathbb{Z}}
\newcommand{\len}{\ell}
\newcommand{\meet}{\vee}
\newcommand{\opi}{\overline{\pi}}
\newcommand{\rad}{\operatorname{rad}}

\newcommand{\sg}[1][n]{{\mathfrak{S}_{#1}}}
\newcommand{\suchthat}{\mid}
\newcommand{\tMonoid}{M}

\newcommand{\N}{\mathbb{N}}

\newcommand{\RR}{\mathcal{R}}
\newcommand{\LL}{\mathcal{L}}
\newcommand{\JJ}{\mathcal{J}}
\newcommand{\HH}{\mathcal{H}}
\newcommand{\BB}{\mathcal{B}}
\newcommand{\KK}{\mathcal{K}}

\newcommand{\lfix}[1]{\operatorname{lfix}(#1)}
\newcommand{\rfix}[1]{\operatorname{rfix}(#1)}
\newcommand{\idempMon}[1][\tMonoid]{E(#1)} % idempotent of a monoid

\newcommand{\raut}[1]{\operatorname{rAut}(#1)}
\newcommand{\laut}[1]{\operatorname{lAut}(#1)}

\newcommand{\biheckemonoid}[1][]{M}

\newcommand{\OR}{\mathcal{OR}}
\newcommand{\NDPF}{{\operatorname{NDPF}}}
\newcommand{\unitribool}{\mathcal U}
\newcommand{\lex}{{\operatorname{lex}}}
\newcommand{\irr}[1][\tMonoid]{\operatorname{Irred}(#1)}
\newcommand{\cirr}[1][\tMonoid]{\operatorname{c-Irred}(#1)}
\newcommand{\quiv}[1][\tMonoid]{\operatorname{Q}(#1)}
\newcommand{\edge}{\!\!\rightarrow\!\!}

\usepackage{ifthen}
\newboolean{draft}
\setboolean{draft}{true}
\ifdraft
\newcommand{\TODO}[2][To do: ]{\textcolor{red}{\textbf{#1#2}}}
\else
\newcommand{\TODO}[2][]{}
\fi

\begin{document}

\title{On the representation theory of finite $\JJ$-trivial monoids}

\author{Tom Denton}
\address{Department of Mathematics, University of California, One
 Shields Avenue, Davis, CA 95616, U.S.A.}
\email{sdenton@math.ucdavis.edu}

\author{Florent Hivert}
\address{LITIS (EA 4108), Universit\'e de Rouen,
  Avenue de l'Universit\'e  BP12
  76801 Saint-Etienne du Rouvray, France and
  Institut Gaspard Monge (UMR 8049), France}
\email{florent.hivert@univ-rouen.fr}

\author{Anne Schilling}
\address{Department of Mathematics, University of California, One
 Shields Avenue, Davis, CA 95616, U.S.A.}
\email{anne@math.ucdavis.edu}

\author{Nicolas M.~Thi\'ery}
\address{Univ Paris-Sud, Laboratoire de Math\'ematiques d'Orsay,
  Orsay, F-91405; CNRS, Orsay, F-91405, France}
\email{Nicolas.Thiery@u-psud.fr}

\maketitle

\begin{abstract}
In 1979, Norton showed that the representation theory of the $0$-Hecke
algebra admits a rich combinatorial description. Her constructions
rely heavily on some triangularity property of the product, but do not
use explicitly that the $0$-Hecke algebra is a monoid algebra.

The thesis of this paper is that considering the general setting of
monoids admitting such a triangularity, namely $\JJ$-trivial monoids,
sheds further light on the topic. This is a step in an ongoing effort
to use representation theory to automatically extract combinatorial
structures from (monoid) algebras, often in the form of posets and
lattices, both from a theoretical and computational point of view, and
with an implementation in Sage.

Motivated by ongoing work on related monoids associated to Coxeter
systems, and building on well-known results in the semi-group
community (such as the description of the simple modules or the
radical), we describe how most of the data associated to the
representation theory (Cartan matrix, quiver) of the algebra of any
$\JJ$-trivial monoid $M$ can be expressed combinatorially by counting
appropriate elements in $M$ itself. As a consequence, this data does
not depend on the ground field and can be calculated in $O(n^2)$, if
not $O(nm)$, where $n=|M|$ and $m$ is the number of generators. Along
the way, we construct a triangular decomposition of the identity into
orthogonal idempotents, using the usual Möbius inversion formula in
the semi-simple quotient (a lattice), followed by an algorithmic
lifting step.

Applying our results to the $0$-Hecke algebra (in all finite types),
we recover previously known results and additionally provide an
explicit labeling of the edges of the quiver. We further explore
special classes of $\JJ$-trivial monoids, and in particular monoids of
order preserving regressive functions on a poset, generalizing known
results on the monoids of nondecreasing parking functions.
\end{abstract}

\tableofcontents

\section{Introduction}

The representation theory of the $0$-Hecke algebra (also called
\emph{degenerate Hecke algebra}) was first studied by P.-N.~Norton~\cite{Norton.1979}
in type A and expanded to other types by
Carter~\cite{Carter.1986}. Using an analogue of Young symmetrizers, they
describe the simple and indecomposable projective modules together with the
Cartan matrix. An interesting combinatorial application was then found by Krob
and Thibon~\cite{Krob_Thibon.NCSF4.1997} who explained how induction and
restriction of these modules gives an interpretation of the products and
coproducts of the Hopf algebras of noncommutative symmetric functions and
quasi-symmetric functions. Two other important steps were further made by
Duchamp--Hivert--Thibon~\cite{Duchamp_Hivert_Thibon.2002} for type $A$ and
Fayers~\cite{Fayers.2005} for other types, using the Frobenius structure to
get more results, including a description of the Ext-quiver.  More recently, a
family of minimal orthogonal idempotents was described
in~\cite{Denton.2010.FPSAC,Denton.2010}. Through divided difference (Demazure
operator), the $0$-Hecke algebra has a central role in Schubert calculus and
also appeared has connection with $K$-theory
\cite{Demazure.1974,Lascoux.2001,Lascoux.2003,Miller.2005,
  Buch_Kresch_Shimozono.2008,Lam_Schilling_Shimozono.2010}.

Like several algebras whose representation theory was studied in
recent years in the algebraic combinatorics community (such as degenerated
left regular bands, Solomon-Tits algebras, ...), the $0$-Hecke
algebra is the algebra of a finite monoid endowed with special
properties. Yet this fact was seldom used (including by the authors),
despite a large body of literature on finite semi-groups, including
representation theory
results~\cite{Putcha.1996,Putcha.1998,Saliola.2007,Saliola.2008,Margolis_Steinberg.2008,Schocker.2008,Steinberg.2006.Moebius,Steinberg.2008.MoebiusII,AlmeidaMargolisVolkov05,Almeida_Margolis_Steinberg_Volkov.2009,Ganyushkin_Mazorchuk_Steinberg.2009,Izhakian.2010.SemigroupsRepresentationsOverSemirings}. From
these, one can see that much of the representation theory of a
semi-group algebra is combinatorial in nature (provided the
representation theory of groups is known).  One can expect, for
example, that for aperiodic semi-groups (which are semi-groups which
contain only trivial subgroups) most of the numerical information
(dimensions of the simple/projective indecomposable modules,
induction/restriction constants, Cartan matrix) can be
computed without using any linear algebra.
In a monoid with partial inverses, one finds (non-trivial) local groups and 
an understanding of the representation theory of these groups is necessary 
for the full representation theory of the monoid.  
In this sense, the notion of aperiodic monoids is orthogonal to that 
of groups as they contain only trivial group-like structure (there are no 
elements with partial inverses).  On the same token, their 
representation theory is orthogonal to that of groups.

The main goal of this paper is to complete this program for the class of
$\JJ$-trivial monoids (a monoid $M$ is \emph{$\JJ$-trivial} provided
that there exists a partial ordering $\leq$ on $M$ such that for all
$x,y\in M$, one has $xy \leq x$ and $xy \leq y$). In this case, we
show that most of the combinatorial data of the representation theory,
including the Cartan matrix and the quiver can be
expressed by counting particular elements in the monoid itself.
A second goal is to provide a self-contained introduction to the
representation theory of finite monoids, targeted at the algebraic
combinatorics audience, and focusing on the simple yet rich case of
$\JJ$-trivial monoids.

The class of $\JJ$-trivial monoids is by itself an active subject of
research (see
e.g.~\cite{Straubing.Therien.1985,Henckell_Pin.2000,Vernitski.2008}),
and contains many monoids of interest, starting with the $0$-Hecke
monoid.  Another classical $\JJ$-trivial monoid is that of
nondecreasing parking functions, or monoid of order preserving
regressive functions on a chain.
Hivert and Thiéry~\cite{Hivert_Thiery.HeckeSg.2006,Hivert_Thiery.HeckeGroup.2007}
showed that it is a natural quotient of the $0$-Hecke monoid and used
this fact to derive its complete representation theory. It is also a
quotient of Kiselman's monoid which is studied
in~\cite{Kudryavtseva_Mazorchuk.2009} with some representation theory
results. Ganyushkin and Mazorchuk~\cite{Ganyushkin_Mazorchuk.2010}
pursued a similar line with a larger family of quotients of both the
$0$-Hecke monoid and Kiselman's monoid.

The extension of the program to larger classes of monoids, like
$\RR$-trivial or aperiodic monoids, is the topic of a forthcoming
paper. Some complications
necessarily arise since the simple modules are not necessarily 
one-dimensional in the latter case.
The approach taken there is to
suppress the dependence upon specific properties of orthogonal
idempotents. Following a complementary line, Berg, Bergeron, Bhargava,
and Saliola~\cite{Berg_Bergeron_Bhargava_Saliola.2010} have very
recently provided a construction for a decomposition of the identity
into orthogonal idempotents for the class of $\RR$-trivial monoids.

The paper is arranged as follows. In Section~\ref{sec:bgnot} we recall
the definition of a number of classes of monoids, including the
$\JJ$-trivial monoids, define some running examples of $\JJ$-trivial
monoids, and establish notation.

In Section~\ref{sec:jrep} we establish the promised results on the
representation theory of $\JJ$-trivial monoids, and illustrates them
on several examples including the $0$-Hecke monoid.  We describe the
radical, construct combinatorial models for the projective and simple
modules, give a lifting construction to obtain orthogonal idempotents,
and describe the Cartan matrix and the quiver, with an explicit
labelling of the edges of the latter. We briefly comment on the
complexity of the algorithms to compute the various pieces of
information, and their implementation in \texttt{Sage}.
All the constructions and proofs involve only combinatorics in the
monoid or linear algebra with unitriangular matrices. Due to this, the
results do not depend on the ground field $\K$. In fact, we have
checked that all the arguments pass to $\K=\ZZ$ and therefore to any
ring (note however that the definition of the quiver that we took
comes from~\cite{Auslander.Reiten.Smaloe.1997}, where it is assumed
that $\K$ is a field). It sounds likely that the theory would apply
mutatis-mutandis to semi-rings, in the spirit
of~\cite{Izhakian.2010.SemigroupsRepresentationsOverSemirings}.

Finally, in Section~\ref{sec:NPDF}, we examine the monoid of order
preserving regressive functions on a poset $P$, which generalizes the
monoid of nondecreasing parking functions on the set $\{1, \ldots,
N\}$. We give combinatorial constructions for idempotents in the
monoid and also prove that the Cartan matrix is upper triangular. In
the case where $P$ is a meet semi-lattice (or, in particular, a
lattice), we establish an idempotent generating set for the monoid,
and present a conjectural recursive formula for orthogonal idempotents
in the algebra.

\subsection{Acknowledgments}

We would like to thank Chris Berg, Nantel Bergeron, Sandeep Bhargava,
Sara Billey, Jean-Éric Pin, Franco Saliola, and Benjamin Steinberg for
enlightening discussions.  We would also like to thank the referee for
detailed reading and many remarks that improved the paper.
This research was driven by computer exploration, using the
open-source mathematical software \texttt{Sage}~\cite{Sage} and its
algebraic combinatorics features developed by the
\texttt{Sage-Combinat} community~\cite{Sage-Combinat}, together with
the \texttt{Semigroupe} package by Jean-Éric Pin~\cite{Semigroupe}.

TD and AS would like to thank the Universit\'e Paris Sud, Orsay for hospitality.
NT would like to thank the Department of Mathematics at UC Davis for hospitality.
TD was in part supported by NSF grants DMS--0652641, DMS--0652652, by VIGRE NSF
grant DMS--0636297, and by a Chateaubriand fellowship from the French Embassy in the US.
FH was partly supported by ANR grant 06-BLAN-0380.
AS was in part supported by NSF grants DMS--0652641, DMS--0652652, and DMS--1001256.
NT was in part supported by NSF grants DMS--0652641, DMS--0652652.

%%%%%%%%%%%%%%%%%%%%%%%%%%%%%%%%%%%%%%%%%%%%%%%%%%%%%%%%%%%%%%%%%%%%%%%%%%%%%%
\section{Background and Notation}
\label{sec:bgnot}
%%%%%%%%%%%%%%%%%%%%%%%%%%%%%%%%%%%%%%%%%%%%%%%%%%%%%%%%%%%%%%%%%%%%%%%%%%%%%% 

A \emph{monoid} is a set $M$ together with a binary operation $\cdot :
M\times M \to M$ such that we have \emph{closure} ($x \cdot y\in M$
for all $x,y \in M$), \emph{associativity} ( $(x\cdot y) \cdot z = x
\cdot ( y \cdot z)$ for all $x,y,z \in M$), and the existence of an
\emph{identity} element $1\in M$ (which satistfies $1\cdot x = x \cdot
1 = x$ for all $x\in M$). 
In this paper, unless explicitly mentioned, all monoids are \emph{finite}.
We use the convention that $A\subseteq B$ denotes $A$ a subset of $B$, and
$A\subset B$ denotes $A$ a proper subset of $B$.

Monoids come with a far richer diversity of features than groups, but
collections of monoids can often be described as \emph{varieties}
satisfying a collection of algebraic identities and closed under
subquotients and finite products (see e.g.~\cite{Pin.1986,Pin.2009}
or~\cite[Chapter VII]{Pin.2009}).  Groups are an example of a variety
of monoids, as are all of the classes of monoids described in this
paper.  In this section, we recall the basic tools for monoids, and
describe in more detail some of the varieties of monoids that are
relevant to this paper. A summary of those is given in
Figure~\ref{fig.monoids}.

\begin{figure}
  \includegraphics[scale=0.75]{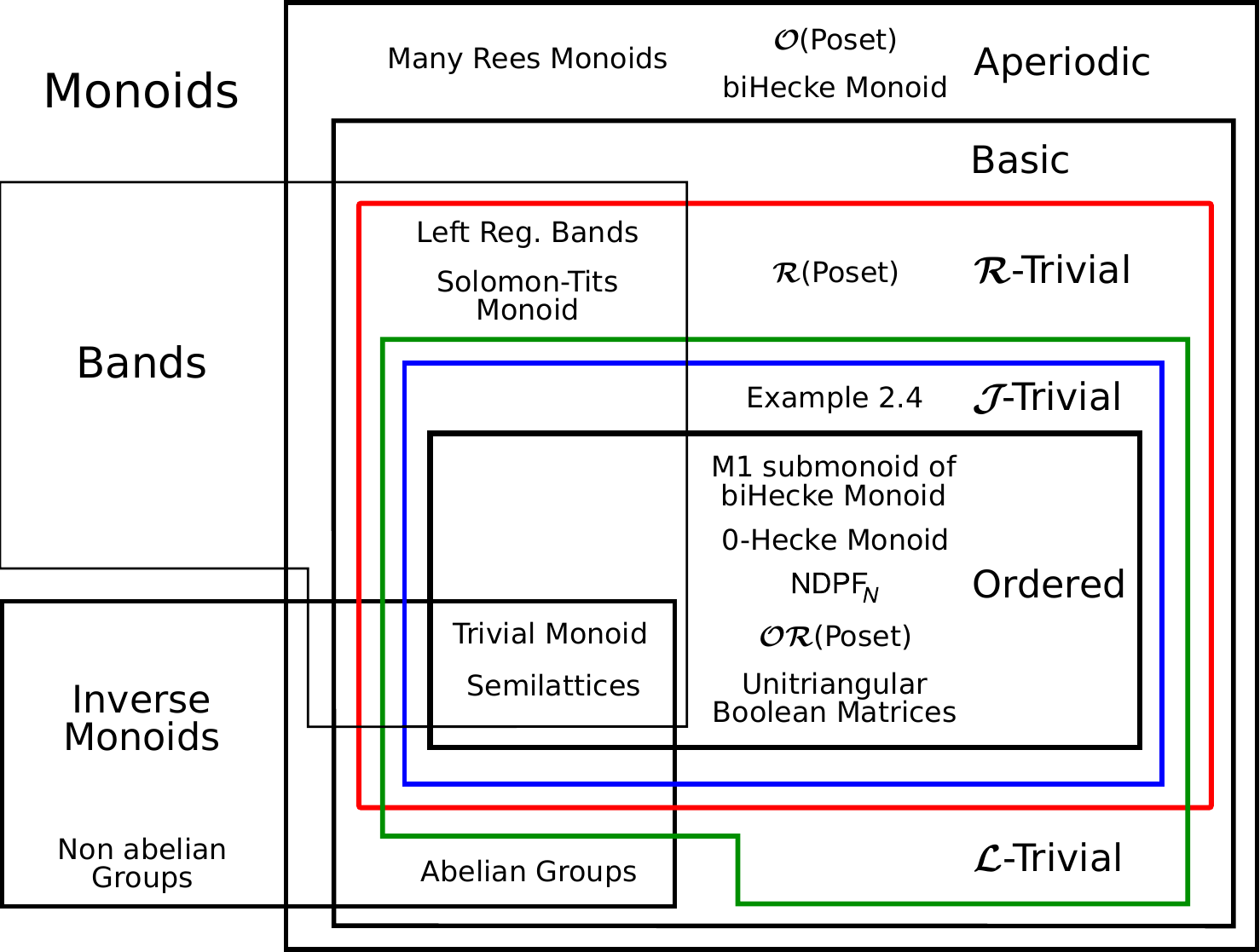}
  \caption{Classes of finite monoids, with examples}
  \label{fig.monoids}
\end{figure}

In 1951 Green introduced several preorders on monoids which are essential for the study of
their structures (see for example~\cite[Chapter V]{Pin.2009}). Let $M$ be a monoid and 
define $\le_\RR, \le_\LL, \le_\JJ, \le_\HH$ for $x,y\in M$ as follows:
\begin{equation*}
\begin{split}
	&x \le_\RR y \quad \text{if and only if $x=yu$ for some $u\in M$}\\
	&x \le_\LL y \quad \text{if and only if $x=uy$ for some $u\in M$}\\
	&x \le_\JJ y \quad \text{if and only if $x=uyv$ for some $u,v\in M$}\\
	&x \le_\HH y \quad \text{if and only if $x\le_\RR y$ and $x\le_\LL y$.}
\end{split}
\end{equation*}
These preorders give rise to equivalence relations:
\begin{equation*}
\begin{split}
	&x \; \RR \; y \quad \text{if and only if $xM = yM$}\\
	&x \; \LL \; y \quad \text{if and only if $Mx = My$}\\
	&x \; \JJ \; y \quad \text{if and only if $MxM = MyM$}\\
	&x \; \HH \; y \quad \text{if and only if $x \; \RR \; y$ and $x \; \LL \; y$.}
\end{split}
\end{equation*}

We further add the relation $\le_\BB$ (and its associated equivalence
relation $\BB$) defined as the finest preorder such that $x\leq_\BB
1$, and
\begin{equation}
\label{equation.bruhat}
  \text{$x\le_\BB y$ implies that $uxv\le_\BB uyv$ for all $x,y,u,v\in M$.}
\end{equation}
(One can view $\le_\BB$ as the intersection of all preorders with the above
property; there exists at least one such preorder, namely $x\le y$ for all $x,y\in M$).

Beware that $1$ is the largest element of these (pre)-orders. This is
the usual convention in the semi-group community, but is the converse
convention from the closely related notions of left/right/Bruhat order
in Coxeter groups.

\begin{definition}
A monoid $M$ is called $\KK$-\emph{trivial} if all $\KK$-classes are
of cardinality one, where $\KK\in \{\RR,\LL,\JJ,\HH,\BB\}$.
\end{definition}

An equivalent formulation of $\KK$-triviality is given in terms of
\emph{ordered} monoids. A monoid $M$ is called:
\begin{equation*}
\begin{aligned}
  & \text{\emph{right ordered}} && \text{if $xy\le x$ for all $x,y\in M$}\\
  & \text{\emph{left ordered}} && \text{if $xy\le y$ for all $x,y\in M$}\\
  & \text{\emph{left-right ordered}} && \text{if $xy\leq x$ and $xy\leq y$ for all $x,y\in M$}\\
  & \text{\emph{two-sided ordered}} && \text{if $xy=yz \leq y$ for all $x,y,z\in M$ with $xy=yz$}\\
  & \text{\emph{ordered with $1$ on top}} && \text{if $x\leq 1$ for all $x\in M$, and $x\le y$}\\ 
  & && \text{implies $uxv\le uyv$ for all $x,y,u,v\in M$}
\end{aligned}
\end{equation*}
for some partial order $\le$ on $M$.

\begin{proposition}
  \label{proposition.ordered}
  $M$ is right ordered (resp. left ordered, left-right ordered, two-sided ordered,
  ordered with $1$ on top) if and only if $M$ is $\RR$-trivial
  (resp. $\LL$-trivial, $\JJ$-trivial, $\HH$-trivial, $\BB$-trivial).

  When $M$ is $\KK$-trivial for $\KK\in \{\RR,\LL,\JJ,\HH,\BB\}$, then
  $\leq_\KK$ is a partial order, \emph{called
    $\KK$-order}. Furthermore, the partial order $\le$ is finer than
  $\le_\KK$: for any $x, y\in M$, $x \le_\KK y$ implies $x \leq y$.
\end{proposition}
\begin{proof}
  We give the proof for right-order as the other cases can be proved in a similar fashion.

  Suppose $M$ is right ordered and that $x,y\in M$ are in the same
  $\RR$-class. Then $x=ya$ and $y=xb$ for some $a,b\in M$. This
  implies that $x\leq y$ and $y\leq x$ so that $x=y$.

  Conversely, suppose that all $\RR$-classes are singletons. Then
  $x\le_\RR y$ and $y\le_\RR x$ imply that $x=y$, so that the
  $\RR$-preorder turns into a partial order. Hence $M$ is
  right ordered using $xy \le_\RR x$.
\end{proof}

\subsection{Aperiodic and $\RR$-trivial monoids}

The class of $\HH$-trivial monoids coincides with that of
\emph{aperiodic} monoids (see for example~\cite[Proposition
4.9]{Pin.2009}): a monoid is called \emph{aperiodic} if for any $x\in
M$, there exists some positive integer $N$ such that
$x^{N}=x^{N+1}$. The element $x^\omega := x^{N}=x^{N+1}=x^{N+2}=\cdots$
is then an idempotent (the idempotent $x^\omega$ can in fact be
defined for any element of any monoid~\cite[Chapter VI.2.3]{Pin.2009},
even infinite monoids; however, the period $k$ such that $x^N = x^{N+k}$ need no
longer be $1$). We write $\idempMon := \{x^\omega \mid x\in M\}$
for the set of idempotents of $M$.

Our favorite example of a monoid which is aperiodic, but not
$\RR$-trivial, is the biHecke monoid studied
in~\cite{Hivert_Schilling_Thiery.BiHeckeMonoid.2010,Hivert_Schilling_Thiery.BiHeckeMonoidRepresentation.2010}. 
This is the submonoid of functions from a finite Coxeter group $W$ to itself
generated simultaneously by the elementary bubble sorting and
antisorting operators $\overline{\pi}_i$ and $\pi_i$
\begin{equation}
\label{equation.bihecke}
   \biheckemonoid(W) :=
   \langle \pi_1, \pi_2, \ldots, \pi_n, \opi_1, \opi_2, \ldots, \opi_n \rangle\,.
\end{equation}
See~\cite[Definition 1.1]{Hivert_Schilling_Thiery.BiHeckeMonoid.2010}
and~\cite[Proposition 3.8]{Hivert_Schilling_Thiery.BiHeckeMonoid.2010}.

The smaller class of $\RR$-trivial monoids coincides with the class of
so-called \emph{weakly ordered monoids} as defined by
Schocker~\cite{Schocker.2008}. Also, via the right regular
representation, any $\RR$-trivial monoid can be represented as a monoid
of regressive functions on some finite poset $P$ (a function $f: P\to
P$ is called \emph{regressive} if $f(x) \le x$ for every $x\in P$);
reciprocally any such monoid is $\RR$-trivial.  We now present an
example of a monoid which is $\RR$-trivial, but not $\JJ$-trivial.
\begin{example}
  Take the free left regular band $\mathcal{B}$ generated by two idempotents
  $a,b$.  Multiplication is given by concatenation taking into account the
  idempotent relations, and then selecting only the two left factors (see for
  example~\cite{Saliola.2007}).  So $\mathcal{B} = \{1,a,b,ab,ba\}$ and
  $1\mathcal{B}=\mathcal{B}$, $a\mathcal{B} = \{a,ab\}$, $b\mathcal{B} =\{b,
  ba\}$, $ab \mathcal{B} = \{ab\}$, and $ba\mathcal{B} = \{ba\}$. This shows
  that all $\RR$-classes consist of only one element and hence $\mathcal{B}$
  is $\RR$-trivial.

  On the other hand, $\mathcal{B}$ is not $\LL$-trivial since
  $\{ab,ba\}$ forms an $\LL$-class since $b\cdot ab = ba$ and $a\cdot
  ba = ab$. Hence $\mathcal{B}$ is also not $\JJ$-trivial.
\end{example}

\subsection{$\JJ$-trivial monoids}

The most important for our paper is the class of $\JJ$-trivial
monoids.  In fact, our main motivation stems from the fact that the
submonoid $M_1=\{ f\in M \mid f(1) =1\} $ of the biHecke monoid $M$
in~\eqref{equation.bihecke} of functions that fix the identity, is
$\JJ$-trivial
(see~\cite[Corollary~4.2]{Hivert_Schilling_Thiery.BiHeckeMonoid.2010}
and~\cite{Hivert_Schilling_Thiery.BiHeckeMonoidRepresentation.2010}).

\begin{example} \label{example.J_trivial}
The following example of a $\JJ$-trivial monoid is given in~\cite{Straubing.Therien.1985}. 
Take $M=\{1,x,y,z,0\}$ with relations 
$x^2=x$, $y^2=y$, $xz=zy=z$, and all other products are equal to $0$. Then $M1M = M$, 
$MxM=\{x,z,0\}$, $MyM=\{y,z,0\}$, $MzM=\{z,0\}$, and $M0M=\{0\}$, which shows that $M$ is 
indeed $\JJ$-trivial. Note also that $M$ is left-right ordered with the order $1>x>y>z>0$,
which by Proposition~\ref{proposition.ordered} is equivalent to $\JJ$-triviality.
\end{example}

\subsection{Ordered monoids (with $1$ on top)}
Ordered monoids $M$ with $1$ on top form a subclass of $\JJ$-trivial
monoids.  To see this suppose that $x,y\in M$ are in the same
$\RR$-class, that is $x=ya$ and $y=xb$ for some $a,b\in M$. Since
$a\le 1$, this implies $x=ya \le y$ and $y=xb\le x$ so that
$x=y$. Hence $M$ is $\RR$-trivial. By analogous arguments, $M$ is also
$\LL$-trivial.  Since $M$ is finite, this implies that $M$ is
$\JJ$-trivial (see~\cite[Chapter V, Theorem 1.9]{Pin.2009}).

The next example shows that ordered monoids with 1 on top form a proper
subclass of $\JJ$-trivial monoids.
\begin{example}
The monoid $M$ of Example~\ref{example.J_trivial} is not ordered.
To see this suppose that $\le$ is an order on $M$ with maximal element $1$. 
The relation $y\le 1$ implies $0=z^2\le z = xzy \le xy =0$ which contradicts $z\neq 0$.
\end{example}

It was shown by Straubing and Th\'erien~\cite{Straubing.Therien.1985} and Henckell and 
Pin~\cite{Henckell_Pin.2000} that every $\JJ$-trivial monoid is a quotient of an ordered
monoid with $1$ on top.

In the next two subsections we present two important examples of
ordered monoids with $1$ on top: the $0$-Hecke monoid and the monoid
of regressive order preserving functions, which generalizes nondecreasing parking functions.

\subsection{$0$-Hecke monoids}
Let $W$ be a finite Coxeter group. It has a presentation
\begin{equation}
  W = \langle\, s_i \; \text{for} \; i\in I\ \suchthat\  (s_is_j)^{m(s_i,s_j)},\ \forall i,j\in I\,\rangle\,,
\end{equation}
where $I$ is a finite set, $m(s_i,s_j) \in \{1,2,\dots,\infty\}$, and $m(s_i,s_i)=1$.
The elements $s_i$ with $i\in I$ are called \emph{simple reflections}, and the
relations can be rewritten as:
\begin{equation}
  \begin{alignedat}{2}
    s_i^2 &=1 &\quad& \text{ for all $i\in I$}\,,\\
    \underbrace{s_is_js_is_js_i \cdots}_{m(s_i,s_j)} &=
    \underbrace{s_js_is_js_is_j\cdots}_{m(s_i,s_j)} && \text{ for all $i,j\in I$}\, ,
  \end{alignedat}
\end{equation}
where $1$ denotes the identity in $W$. An expression $w=s_{i_1}\cdots s_{i_\ell}$ for $w\in W$ 
is called \emph{reduced} if it is of minimal length $\ell$.
See~\cite{bjorner_brenti.2005, humphreys.1990} for further details on Coxeter groups.

The Coxeter group of type $A_{n-1}$ is the symmetric group $\sg[n]$ with generators
$\{s_1,\dots,s_{n-1}\}$ and relations:
\begin{equation}
  \begin{alignedat}{2}
    s_i^2           & = 1                &     & \text{ for } 1\leq i\leq n-1\,,\\
    s_i s_j         & = s_j s_i            &     & \text{ for } |i-j|\geq2\,, \\
    s_i s_{i+1} s_i & = s_{i+1} s_i s_{i+1} &\quad& \text{ for } 1\leq i\leq n-2\,;
  \end{alignedat}
\end{equation}
the last two relations are called the \emph{braid relations}.

\begin{definition}[\textbf{$0$-Hecke monoid}]
The $0$-Hecke monoid $H_0(W) = \langle \pi_i \mid i \in I \rangle$ of a Coxeter group $W$ 
is generated by the \emph{simple projections} $\pi_i$ with relations
\begin{equation}
  \begin{alignedat}{2}
    \pi_i^2 &=\pi_i &\quad& \text{ for all $i\in I$,}\\
    \underbrace{\pi_i\pi_j\pi_i\pi_j\cdots}_{m(s_i,s_{j})} &=
    \underbrace{\pi_j\pi_i\pi_j\pi_i\cdots}_{m(s_i,s_{j})} && \text{ for all $i,j\in I$}\ .
  \end{alignedat}
\end{equation}
Thanks to these relations, the elements of $H_0(W)$ are canonically
indexed by the elements of $W$ by setting $\pi_w :=
\pi_{i_1}\cdots\pi_{i_k}$ for any reduced word $i_1 \dots i_k$ of $w$.
\end{definition}

\emph{Bruhat order} is a partial order defined on any Coxeter group $W$
and hence also the corresponding $0$-Hecke monoid $H_0(W)$.  Let
$w=s_{i_1} s_{i_2} \cdots s_{i_\ell}$ be a reduced expression for
$w\in W$. Then, in Bruhat order $\le_B$,
\begin{equation*}
	u\le_B w \quad \begin{array}[t]{l}
	\text{if there exists a reduced expression $u = s_{j_1} \cdots s_{j_k}$}\\
	\text{where $j_1 \ldots j_k$ is a subword of $i_1 \ldots i_\ell$.} \end{array}
\end{equation*}
In Bruhat order, $1$ is the minimal element. Hence, it is not hard to check that, with
reverse Bruhat order, the $0$-Hecke monoid is indeed an ordered monoid with $1$
on top.

In fact, the orders $\le_\LL$, $\le_\RR$, $\le_\JJ$, $\le_\BB$ on
$H_0(W)$ correspond exactly to the usual (reversed) left, right,
left-right, and Bruhat order on the Coxeter group $W$.

\subsection{Monoid of regressive order preserving functions}

For any partially ordered set $P$, there is a particular $\JJ$-trivial monoid
which has some very nice properties and that we investigate further in
Section~\ref{sec:NPDF}. Notice that we use the right action in this paper, so
that for $x\in P$ and a function $f:P\to P$ we write $x.f$ for the value of
$x$ under $f$.

\begin{definition}[\textbf{Monoid of regressive order preserving functions}]
  Let $(P, \leq_P)$ be a poset. The set $\OR(P)$ of functions $f: P \to P$ which are
  \begin{itemize}
  \item \emph{order preserving}, that is, for all $x,y\in P,\ x\leq_P y$ implies
    $x.f\leq_P y.f$
  \item \emph{regressive}, that is, for all $x\in P$ one has $x.f \leq_P x$
  \end{itemize}
  is a monoid under composition.
\end{definition}
\begin{proof}
  It is trivial that the identity function is order preserving and regressive and that
  the composition of two order preserving and regressive functions is as well.
\end{proof}

According to~\cite[14.5.3]{Ganyushkin_Mazorchuk.2009}, not much is
known about these monoids.

When $P$ is a chain on $N$ elements, we obtain the monoid $\NDPF_N$ of
nondecreasing parking functions on the set $\{1, \ldots, N\}$ (see
e.g.~\cite{Solomon.1996}; it also is described under the notation
$\mathcal C_n$ in e.g.~\cite[Chapter~XI.4]{Pin.2009} and, together
with many variants, in~\cite[Chapter~14]{Ganyushkin_Mazorchuk.2009}).
The unique minimal set of generators for $\NDPF_N$ is given by the
family of idempotents $(\pi_i)_{i\in\{1,\dots,n-1\}}$, where each
$\pi_i$ is defined by $(i+1).\pi_i:=i$ and $j.\pi_i:=j$ otherwise. The
relations between those generators are given by:
\begin{gather*}
  \pi_i\pi_j = \pi_j\pi_i \quad \text{ for all $|i-j|>1$}\,,\\
  \pi_i\pi_{i-1}=\pi_i\pi_{i-1}\pi_i=\pi_{i-1}\pi_i\pi_{i-1}\,.
\end{gather*}
It follows that $\NDPF_n$ is the natural quotient of $H_0(\sg[n])$ by
the relation $\pi_i\pi_{i+1}\pi_i = \pi_{i+1}\pi_i$, via the quotient
map $\pi_i\mapsto
\pi_i$~\cite{Hivert_Thiery.HeckeSg.2006,Hivert_Thiery.HeckeGroup.2007,
  Ganyushkin_Mazorchuk.2010}. Similarly, it is a natural quotient of
Kiselman's
monoid~\cite{Ganyushkin_Mazorchuk.2010,Kudryavtseva_Mazorchuk.2009}.

To see that $\OR(P)$ is indeed a subclass of ordered monoids with $1$
on top, note that we can define a partial order by saying $f\le g$ for
$f,g\in \OR(P)$ if $x.f \le_P x.g$ for all $x\in P$. By
regressiveness, this implies that $f\le \id$ for all $f\in \OR(P)$ so
that indeed $\id$ is the maximal element. Now take $f,g,h \in \OR(P)$
with $f\le g$.  By definition $x.f \le_P x.g$ for all $x\in P$ and
hence by the order preserving property $(x.f).h \le_P (x.g).h$, so
that $fh\le gh$. Similarly since $f\le g$, $(x.h).f \le_P (x.h).g$ so
that $hf\le hg$. This shows that $\OR(P)$ is ordered.

The submonoid $M_1$ of the biHecke monoid~\eqref{equation.bihecke},
and $H_0(W)\subset M_1$, are submonoids of the monoid of regressive
order preserving functions acting on the Bruhat poset.

\subsection{Monoid of unitriangular Boolean matrices}

Finally, we define the $\JJ$-trivial monoid $\unitribool_n$ of
\emph{unitriangular Boolean matrices}, that is of $n\times n$ matrices
$m$ over the Boolean semi-ring which are unitriangular: $m[i,i]=1$ and
$m[i,j]=0$ for $i>j$. Equivalently (through the adjacency matrix),
this is the monoid of the binary reflexive relations contained in the
usual order on $\{1,\dots,n\}$ (and thus antisymmetric), equipped with
the usual composition of relations. Ignoring loops, it is convenient
to depict such relations by acyclic digraphs admitting $1,\dots,n$ as
linear extension. The product of $g$ and $h$ contains the edges of
$g$, of $h$, as well as the transitivity edges $i\edge k$ obtained
from one edge $i\edge j$ in $g$ and one edge $j\edge k$ in $h$. Hence,
$g^2=g$ if and only if $g$ is transitively closed.

The family of monoids $(\unitribool_n)_n$ (resp. $(\NDPF_n)_n$) plays a
special role, because any $\JJ$-trivial monoid is a subquotient of
$\unitribool_n$ (resp. $\NDPF_n$) for $n$ large
enough~\cite[Chapter~XI.4]{Pin.2009}. In particular, $\NDPF_n$ itself
is a natural submonoid of $\unitribool_n$.
\begin{remark}
  We now demonstrate how $\NDPF_n$ can be realized as a submonoid of relations.
  For simplicity of notation, we consider the monoid $\OR(P)$
  where $P$ is the reversed chain $\{1>\dots>n\}$. Otherwise said,
  $\OR(P)$ is the monoid of functions on the chain $\{1<\dots<n\}$
  which are order preserving and extensive ($x.f\geq x$). Obviously,
  $\OR(P)$ is isomorphic to $\NDPF_n$.

  The monoid $\OR(P)$ is isomorphic to the submonoid of the
  relations $A$ in $\unitribool_n$ such that $i\edge j \in A$ implies
  $k\edge l\in A$ whenever $i\geq k\geq l\geq j$ (in
  the adjacency matrix: $(k,l)$ is to the south-west of $(i,j)$ and
  both are above the diagonal). The isomorphism is given by the map $A
  \mapsto f_A\in \OR(P)$, where
  \begin{equation*}
    u\cdot f_A := \max\{v\suchthat u\,\edge\,v\in A\}\,.
  \end{equation*}
  The inverse bijection $f\in \OR(P) \mapsto A_f\in \unitribool_n$ is given by
  \begin{equation*}
    u\,\edge\,v \in A_f \text{ if and only if } u\cdot f \leq v\,.
  \end{equation*}
  For example, here are the elements of $\OR(\{1>2>3\})$ and the
  adjacency matrices of the corresponding relations in
  $\unitribool_3$:
  \begin{equation*}
    \begin{array}{ccccc}
      \begin{tikzpicture}[->,baseline=(current bounding box.east)]
        \matrix (m) [matrix of math nodes, row sep=.5em, column sep=1.5em]{
          1      & 1 \\
          2      & 2 \\
          3      & 3 \\
        };
        \draw (m-1-1) -> (m-1-2);
        \draw (m-2-1) -> (m-2-2);
        \draw (m-3-1) -> (m-3-2);
      \end{tikzpicture}&
      \begin{tikzpicture}[->,baseline=(current bounding box.east)]
        \matrix (m) [matrix of math nodes, row sep=.5em, column sep=1.5em]{
          1      & 1 \\
          2      & 2 \\
          3      & 3 \\
        };
        \draw (m-1-1) -> (m-2-2);
        \draw (m-2-1) -> (m-2-2);
        \draw (m-3-1) -> (m-3-2);
      \end{tikzpicture}&
      \begin{tikzpicture}[->,baseline=(current bounding box.east)]
        \matrix (m) [matrix of math nodes, row sep=.5em, column sep=1.5em]{
          1      & 1 \\
          2      & 2 \\
          3      & 3 \\
        };
        \draw (m-1-1) -> (m-1-2);
        \draw (m-2-1) -> (m-3-2);
        \draw (m-3-1) -> (m-3-2);
      \end{tikzpicture}&
      \begin{tikzpicture}[->,baseline=(current bounding box.east)]
        \matrix (m) [matrix of math nodes, row sep=.5em, column sep=1.5em]{
          1      & 1 \\
          2      & 2 \\
          3      & 3 \\
        };
        \draw (m-1-1) -> (m-2-2);
        \draw (m-2-1) -> (m-3-2);
        \draw (m-3-1) -> (m-3-2);
      \end{tikzpicture}&
      \begin{tikzpicture}[->,baseline=(current bounding box.east)]
        \matrix (m) [matrix of math nodes, row sep=.5em, column sep=1.5em]{
          1      & 1 \\
          2      & 2 \\
          3      & 3 \\
        };
        \draw (m-1-1) -> (m-3-2);
        \draw (m-2-1) -> (m-3-2);
        \draw (m-3-1) -> (m-3-2);
      \end{tikzpicture}\\\\
      \begin{pmatrix}
        1 & 0 & 0\\
        0 & 1 & 0\\
        0 & 0 & 1\\
      \end{pmatrix} &
      \begin{pmatrix}
        1 & 1 & 0\\
        0 & 1 & 0\\
        0 & 0 & 1\\
      \end{pmatrix} &
      \begin{pmatrix}
        1 & 0 & 0\\
        0 & 1 & 1\\
        0 & 0 & 1\\
      \end{pmatrix} &
      \begin{pmatrix}
        1 & 1 & 0\\
        0 & 1 & 1\\
        0 & 0 & 1\\
      \end{pmatrix} &
      \begin{pmatrix}
        1 & 1 & 1\\
        0 & 1 & 1\\
        0 & 0 & 1\\
      \end{pmatrix} \; .
    \end{array}
  \end{equation*}
\end{remark}

%%%%%%%%%%%%%%%%%%%%%%%%%%%%%%%%%%%%%%%%%%%%%%%%%%%%%%%%%%%%%%%%%%%%%%%%%%%%%%
\section{Representation theory of $\JJ$-trivial monoids}
\label{sec:jrep}
%%%%%%%%%%%%%%%%%%%%%%%%%%%%%%%%%%%%%%%%%%%%%%%%%%%%%%%%%%%%%%%%%%%%%%%%%%%%%% 

In this section we study the representation theory of $\JJ$-trivial
monoids $\tMonoid$, using the $0$-Hecke monoid $H_0(W)$ of a finite
Coxeter group as running example. In Section~\ref{ss.simple.radical}
we construct the simple modules of $\tMonoid$ and derive a description
of the radical $\rad\K\tMonoid$ of the monoid algebra of
$\tMonoid$. We then introduce a star product on the set $\idempMon$ of
idempotents in Theorem~\ref{theorem.star} which makes it into a
semi-lattice, and prove in
Corollary~\ref{corollary.triangular-radical} that the semi-simple
quotient of the monoid algebra $\K\tMonoid/\rad \K\tMonoid$ is the
monoid algebra of $(\idempMon,\star)$. In
Section~\ref{ss.orthogonal.idempotents} we construct orthogonal
idempotents in $\K\tMonoid/\rad \K\tMonoid$ which are lifted to a
complete set of orthogonal idempotents in $\K\tMonoid$ in
Theorem~\ref{theorem.idempotents.lifting} in Section~\ref{ss.lifting}.
In Section~\ref{ss.cartan} we describe the Cartan matrix of
$\tMonoid$. We study several types of factorizations in
Section~\ref{ss.factorizations}, derive a combinatorial description of the quiver 
of $\tMonoid$ in Section~\ref{ss.quiver}, and apply it
in Section~\ref{ss.quiver.examples} to several examples. Finally, in
Section~\ref{subsection.implementation_complexity}, we briefly comment
on the complexity of the algorithms to compute the various pieces of
information, and their implementation in \texttt{Sage}.

\subsection{Simple modules, radical, star product, and semi-simple quotient}
\label{ss.simple.radical}

The goal of this subsection is to construct the simple modules of the
algebra of a $\JJ$-trivial monoid $\tMonoid$, and to derive a
description of its radical and its semi-simple quotient. The proof
techniques are similar to those of Norton~\cite{Norton.1979} for the
$0$-Hecke algebra. However, putting them in the context of
$\JJ$-trivial monoids makes the proofs more transparent. In fact, most
of the results in this section are already known and admit natural
generalizations in larger classes of monoids ($\RR$-trivial, ...). For
example, the description of the radical is a special case of
Almeida-Margolis-Steinberg-Volkov~\cite{Almeida_Margolis_Steinberg_Volkov.2009},
and that of the simple modules
of~\cite[Corollary~9]{Ganyushkin_Mazorchuk_Steinberg.2009}.

Also, the description of the semi-simple quotient is often derived
alternatively from the description of the radical, by noting that it
is the algebra of a monoid which is $\JJ$-trivial and idempotent
(which is equivalent to being a semi-lattice; see
e.g.~\cite[Chapter VII, Proposition 4.12]{Pin.2009}).
\begin{proposition}\label{proposition.simple}
  Let $\tMonoid$ be a $\JJ$-trivial monoid and $x\in \tMonoid$. Let $S_x$ be the
  $1$-dimensional vector space spanned by an element $\epsilon_x$, and define the
  right action of any $y\in \tMonoid$ by
  \begin{equation}
    \epsilon_x y =
    \begin{cases}
      \epsilon_x & \text{if $xy=x$,}\\
      0          & \text{otherwise.}
    \end{cases}
  \end{equation}
  Then $S_x$ is a right $\tMonoid$-module. Moreover, any simple module is isomorphic
  to $S_x$ for some $x \in \tMonoid$ and is in particular one-dimensional.
\end{proposition}

Note that some $S_x$ may be isomorphic to each other, and that the
$S_x$ can be similarly endowed with a left $\tMonoid$-module structure.

\begin{proof}
  Recall that, if $\tMonoid$ is $\JJ$-trivial, then $\leq_\JJ$ is a
  partial order called $\JJ$-order (see
  Proposition~\ref{proposition.ordered}). Let $(x_1, x_2, \ldots,
  x_n)$ be a linear extension of $\JJ$-order, that is an enumeration
  of the elements of $\tMonoid$ such that $x_i \leq_\JJ x_j$ implies
  $i\leq j$. For $0<i\leq
  n$, define $F_i = \K\{x_j \suchthat j\leq i\}$ and set
  $F_0=\{0_\K\}$.  Clearly the $F_i$'s are ideals of $\K\tMonoid$
  such that the sequence
  \begin{equation*}
    F_0 \subset F_1 \subset F_2 \subset \cdots \subset F_{n-1} \subset F_n
  \end{equation*}
  is a composition series for the regular representation
  $F_n=\K\tMonoid$ of $\tMonoid$. Moreover, for any $i>0$, the
  quotient $F_i/F_{i-1}$ is a one-dimensional $\tMonoid$-module
  isomorphic to $S_{x_i}$. Since any simple $\tMonoid$-module must
  appear in any composition series for the regular representation, it
  has to be isomorphic to $F_i/F_{i-1}\cong S_{x_i}$ for some $i$.
\end{proof}

\begin{corollary}
  Let $\tMonoid$ be a $\JJ$-trivial monoid. Then, the quotient of its
  monoid algebra $\K\tMonoid$ by its radical is commutative.
\end{corollary}
Note that the radical $\rad \K\tMonoid$ is not necessarily generated
as an ideal by $\{gh - hg \suchthat g,h\in \tMonoid\}$. For example,
in the commutative monoid $\{1, x, 0\}$ with $x^2=0$, the radical is
$\K (x - 0)$. However, thanks to the following this is true if
$\tMonoid$ is generated by idempotents (see
Corollary~\ref{corollary.rad.idemp}).

The following proposition gives an alternative description of the radical of
$\K\tMonoid$.
\begin{proposition}\label{proposition.basis.radical}
  Let $\tMonoid$ be a $\JJ$-trivial monoid. Then
  \begin{equation}
    \{ x-x^\omega \suchthat x\in \tMonoid\backslash\idempMon \}
  \end{equation}
  is a basis for $\rad\K\tMonoid$.

  Moreover $(S_e)_{e \in \idempMon}$ is a complete set of pairwise non-isomorphic 
  representatives of isomorphism classes of simple
  $\tMonoid$-modules.
\end{proposition}

\begin{proof}
  For any $x,y \in \tMonoid$, either $yx=y$ and then $yx^\omega =y$,
  or $yx <_\JJ y$ and then $yx^\omega <_\JJ y$. Therefore $x-x^\omega$
  is in $\rad\K\tMonoid$ because for any $y$ the product
  $\epsilon_y(x-x^\omega)$ vanishes. Since $x^\omega \leq x$, by
  triangularity with respect to $\JJ$-order, the family
  \begin{equation*}
    \{ x-x^\omega \suchthat x\in \tMonoid\backslash \idempMon\} \cup \idempMon
  \end{equation*}
  is a basis of $\K\tMonoid$. There remains to show that the radical
  is of dimension at most the number of non-idempotents in $\tMonoid$,
  which we do by showing that the simple modules $(S_e)_{e\in
    \idempMon}$ are not pairwise isomorphic. Assume that $S_e$ and
  $S_f$ are isomorphic. Then, since $\epsilon_e e = \epsilon_e$, it
  must be that $\epsilon_e f = \epsilon_e$ so that $ef=e$. Similarly
  $fe=f$, so that $e$ and $f$ are in the same $\JJ$-class and
  therefore equal.
\end{proof}

The following theorem elucidates the structure of the semi-simple quotient of
the monoid algebra $\K\tMonoid$.
\begin{theorem} \label{theorem.star} 
  Let $\tMonoid$ be a $\JJ$-trivial monoid. Define a product $\star$ on
  $\idempMon$ by: 
  \begin{equation}
    e \star f := (e f)^\omega\,.
  \end{equation}
  Then, the restriction of $\leq_\JJ$ on $\idempMon$ is a lattice such
  that
  \begin{equation}
    e \wedge_\JJ f = e \star f\,,
  \end{equation}
  where $e \wedge_\JJ f$ is the meet or infimum of $e$ and $f$ in the
  lattice. In particular $(\idempMon, \star)$ is an idempotent
  commutative $\JJ$-trivial monoid.
\end{theorem}

We start with two preliminary easy lemmas (which are consequences of
e.g.~\cite[Chapter VII, Proposition~4.10]{Pin.2009}).
\begin{lemma}
\label{lemma.idem_factor}
  If $e \in \idempMon$ is such $e = ab$ for some $a,b\in \tMonoid$, then 
  \[
  e=ea=be=ae=eb\,.
  \]
\end{lemma}
\begin{proof}
  For $e\in \idempMon$, one has $e=e^3$ so that $e=eabe$. As a
  consequence, $e\leq_\JJ ea\leq_\JJ e$ and $e\leq_\JJ be\leq_\JJ e$,
  so that $e=ea=be$.  In addition $e=e^2=eab=eb$ and $e=e^2=abe=ae$.
\end{proof}

\begin{lemma}\label{lemma.j.idemp}
  For $e\in \idempMon$ and $y \in \tMonoid$, the following three statements are
  equivalent:
  \begin{equation}
    e \leq_\JJ y, \qquad\qquad e = ey, \qquad\qquad e = ye \;.
  \end{equation}
\end{lemma}
\begin{proof}
  Suppose that $e,y$ are such that $e \leq_\JJ y$. Then $e=ayb$ for some $a,b\in
  \tMonoid$. Applying Lemma~\ref{lemma.idem_factor} we obtain $e=ea=be$ so that 
  $eye = eaybe = eee = e$ since $e\in \idempMon$. A second application of 
  Lemma~\ref{lemma.idem_factor} shows that
  $ey = eye =e$ and $ye = eye = e$.
  The converse implications hold by the definition of $\leq_\JJ$.
\end{proof}

\begin{proof}[Proof of Theorem~\ref{theorem.star}]
  We first show that, for any $e,f \in \idempMon$ the product $e\star
  f$ is the greatest lower bound $e\wedge_\JJ f$ of $e$ and $f$ so that
  the latter exists. It is clear that $(ef)^\omega \leq_\JJ e$ and
  $(ef)^\omega \leq_\JJ f$. Take now $z\in \idempMon$ satisfying
  $z\leq_\JJ e$ and $z\leq_\JJ f$. Applying Lemma~\ref{lemma.j.idemp},
  $z = ze = zf$, and therefore $z = z(ef)^\omega$. Applying
  Lemma~\ref{lemma.j.idemp} backward, $z\leq_\JJ (ef)^\omega$, as
  desired.

  Hence $(\idempMon, \leq_\JJ)$ is a meet semi-lattice with a greatest
  element which is the unit of $\tMonoid$. It is therefore a lattice
  (see e.g.~\cite{Stanley.1999.EnumerativeCombinatorics1,Wikipedia.Lattice}). Since lower bound is a
  commutative associative operation, $(\idempMon, \star)$ is a
  commutative idempotent monoid.
\end{proof}

We can now state the main result of this section.
\begin{corollary} 
  \label{corollary.triangular-radical}
  Let $\tMonoid$ be a $\JJ$-trivial monoid. Then, $(\K\idempMon,
  \star)$ is isomorphic to $\K\tMonoid/\rad\K\tMonoid$ and $\phi: x
  \mapsto x^\omega$ is the canonical algebra morphism associated to
  this quotient.
\end{corollary}
\begin{proof}
  Denote by $\psi\ :\ \K\tMonoid \to\K\tMonoid/\rad\K\tMonoid$
  the canonical algebra morphism. It follows from
  Proposition~\ref{proposition.basis.radical} that, for any $x$
  (idempotent or not), $\psi(x) = \psi(x^\omega)$ and that
  $\{\psi(e)\suchthat e\in \idempMon\}$ is a basis for the
  quotient. Finally, $\star$ coincides with the product in the
  quotient: for any $e,f\in \idempMon$,
  \begin{displaymath}
    \psi(e)\psi(f) = \psi(ef) = \psi((ef)^\omega) = \psi(e\star f)\,.\qedhere
  \end{displaymath}
\end{proof}

\begin{corollary}
  \label{corollary.rad.idemp}
  Let $\tMonoid$ be a $\JJ$-trivial monoid generated by
  idempotents. Then the radical $\rad\K\tMonoid$ of its monoid algebra is
  generated as an ideal by 
  \begin{equation}
    \{gh - hg \suchthat g,h\in \tMonoid\}\,.
  \end{equation}
\end{corollary}
\begin{proof}
  \newcommand{\Com}{\mathcal{C}}
  Denote by $\Com$ the ideal generated by $\{gh - hg \suchthat g,h\in
  \tMonoid\}$. Since $\rad\K\tMonoid$ is the linear span of
  $(x-x^\omega)_{x\in\tMonoid}$, it is sufficient to show that for any
  $x\in\tMonoid$ one has $x\equiv x^2 \pmod \Com$. Now write $x=e_1\cdots e_n$
  where $e_i$ are all idempotent. Then,
  \begin{equation*}
    x \equiv e_1^2\cdots e_n^2 \equiv e_1\cdots e_ne_1\cdots e_n 
    \equiv x^2 \pmod \Com\,.\qedhere
  \end{equation*}
\end{proof}

\begin{example}[Representation theory of $H_0(W)$]
\label{example.zero.hecke}
Consider the $0$-Hecke monoid $H_0(W)$ of a finite Coxeter group $W$, with index set $I=\{1, 2, \ldots, n\}$. 
For any $J \subseteq I$, we can consider the parabolic submonoid $H_0(W_J)$ generated by
$\{\pi_i \mid i\in J\}$.  Each parabolic submonoid contains a unique longest element $\longest_J$.  
The collection $\{\longest_J \mid J\subseteq I\}$ is exactly the set of idempotents in $H_0(W)$.

For each $i\in I$, we can construct the \emph{evaluation maps} $\Phi_i^+$ and $\Phi_i^-$ defined 
on generators by:
\begin{eqnarray*}
\Phi_i^+ &:& \mathbb{C}H_0(W) \rightarrow \mathbb{C}H_0(W_{I\setminus \{i\}}) \\
\Phi_i^+(\pi_j) &=&     \begin{cases}
      1          & \text{if $i=j$,}\\
      \pi_j & \text{if $i \neq j$,}
    \end{cases}
\end{eqnarray*}
and
\begin{eqnarray*}
\Phi_i^- &:& \mathbb{C}H_0(W) \rightarrow \mathbb{C}H_0(W_{I\setminus \{i\}}) \\
\Phi_i^-(\pi_j) &=&     \begin{cases}
      0          & \text{if $i=j$,}\\
      \pi_j & \text{if $i \neq j$.}
    \end{cases}
\end{eqnarray*}
One can easily check that these maps extend to algebra morphisms from 
$H_0(W)\rightarrow H_0(W_{I\setminus \{i\}})$.  For any $J$, define $\Phi_J^+$ as the composition 
of the maps $\Phi_i^+$ for $i\in J$, and define $\Phi_J^-$ analogously (the map $\Phi_J^+$ is 
the \emph{parabolic map} studied by Billey, Fan, and Losonczy~\cite{Billey_Fan_Lsonczy.1999}).  
Then, the simple representations of $H_0(W)$ are given by the maps 
$\lambda_J = \Phi_J^+ \circ \Phi_{\hat{J}}^-$, where $\hat{J}=I\setminus J$.  
This is clearly a one-dimensional representation.
\end{example}

\subsection{Orthogonal idempotents}
\label{ss.orthogonal.idempotents}

We describe here a decomposition of the identity of the semi-simple
quotient into minimal orthogonal idempotents.  We include a proof for
the sake of completeness, though the result is classical. It appears
for example in a combinatorial context in~\cite[Section
3.9]{Stanley.1999.EnumerativeCombinatorics1} and in the context of
semi-groups in~\cite{Solomon.1967,Steinberg.2006.Moebius}.

For $e\in \idempMon$, define
\begin{equation} \label{equation.g}
  g_e:=\sum_{e'\leq_\JJ e} \mu_{e',e} e'\,, 
\end{equation}
where $\mu$ is the M\"obius function of $\leq_\JJ$, so that 
\begin{equation} \label{equation.e}
  e=\sum_{e'\leq_\JJ e} g_{e'}\,.
\end{equation}
\begin{proposition}
  The family $\{g_e\suchthat e\in \idempMon\}$ is the unique maximal
  decomposition of the identity into orthogonal idempotents for $\star$ that
  is in $\K\tMonoid/\rad\K\tMonoid$.
\end{proposition}
\begin{proof}
  First note that $1_\tMonoid = \sum_e g_e$ by~\eqref{equation.e}.

  Consider now the new product $\bullet$ on $\K\idempMon
  = \K\{g_e \suchthat e\in \idempMon\}$ defined by
  $g_u\bullet g_v = \delta_{u,v} g_u$.
  Then,
  \begin{displaymath}
    u \bullet v = \sum_{u'\leq_\JJ u} g_{u'} \bullet \sum_{v'\leq_\JJ v}
    g_{v'} = \sum_{w'\leq u \wedge_\JJ v} g_{w'} =
    u \wedge_\JJ v = u \star v\,.
  \end{displaymath}
  Hence the product $\bullet$ coincides with $\star$.
  
  Uniqueness follows from semi-simplicity and the fact that all simple modules
  are one-dimensional.
\end{proof}

%%%%%%%%%%%%%%%%%%%%%%%%%%%%%%%%%%%%%%%%%%%%%%%%%%%%%%%%%%%%%%%%%%%%%%%%%%%%%%
\subsection{Lifting the idempotents}
\label{ss.lifting}

In the following we will need a decomposition of the identity in the algebra
of the monoid with some particular properties. The goal of this section is to
construct such a decomposition. The idempotent lifting is a well-known
technique (see~\cite[Chapter 7.7]{curtis_reiner.1962}), however we prove the result from scratch in order to obtain a lifting with particular properties. Moreover, the proof
provided here is very constructive.
\begin{theorem}
\label{theorem.idempotents.lifting}
  Let $\tMonoid$ be a $\JJ$-trivial monoid. There exists a family $(f_e)_{e\in
    \idempMon}$ of elements of $\K\tMonoid$ such that
  \begin{itemize}
  \item $(f_e)$ is a decomposition of the identity of $\K\tMonoid$ into
    orthogonal idempotents:
    \begin{equation}
      1 = \sum_{e\in \idempMon} f_e
      \qquad\text{with}\qquad f_ef_{e'} = \delta_{e,e'} f_e\,.
    \end{equation}
  \item $(f_e)$ is compatible with the semi-simple quotient:
    \begin{equation}
      \phi(f_e) = g_e \quad \text{with $\phi$ as in Corollary~\ref{corollary.triangular-radical}.}
    \end{equation}
  \item $(f_e)$ is uni-triangular with respect to the $\JJ$-order of $\tMonoid$:
    \begin{equation}
      f_e = e + \sum_{x<_\JJ e} c_{x,e} x
    \end{equation}
    for some scalars $c_{x,e}$.
  \end{itemize}
\end{theorem}

This theorem will follow directly from Proposition~\ref{proposition.unitriangular} below.
In the proof, we will use the following proposition:
\begin{proposition}\label{proposition.idempotents.lifting}
  Let $A$ be a finite-dimensional $\K$-algebra and $\phi$ the canonical algebra
  morphism from $A$ to $A/\rad A$. Let $x\in A$ be such that $e=\phi(x)$ is
  idempotent. Then, there exists a polynomial $P\in x\ZZ[x]$ (i.e. without
  constant term) such that $y=P(x)$ is idempotent and $\phi(y) = e$. Moreover,
  one can choose $P$ so that it only depends on the dimension of $A$ (and not
  on $x$ or $A$).
\end{proposition}
Let us start with two lemmas, where we keep the same assumptions as in 
Proposition~\ref{proposition.idempotents.lifting}, namely $x\in A$ such that $\phi(x)=e$
is an idempotent:
\begin{lemma}\label{lemma.x.nilpotent}
  $x(x-1)$ is nilpotent: $(x(x-1))^u = 0$ for some $u$.
\end{lemma}
\begin{proof}
  $e=\phi(x)$ is idempotent so that $e(e-1) = 0$. Hence $x(x-1)\in\rad A$ and
  is therefore nilpotent.
\end{proof}
For any number $a$ denote by $\lceil a\rceil$ the smallest integer larger than
$a$.
\begin{lemma}
  Suppose that $(x(x-1))^u=0$ and define $y := 1 - (1-x^2)^2 = 2x^2-x^4$. Then
  $(y(y-1))^v=0$ with $v=\lceil\frac{u}{2}\rceil$.
\end{lemma}
\begin{proof}
  It suffices to expand and factor
  $y(y-1) = x^2 (x - 1)^2 (x + 1)^2 (x^2 - 2)$.
  Therefore $(y(y-1))^v$ is divisible by $(x(x-1))^u$ and must vanish.
\end{proof}
\begin{proof}[Proof of Proposition~\ref{proposition.idempotents.lifting}]
  Define $y_0:=x$ and $y_{n+1}:= 1-(1-y_n^2)^2$. Then by Lemma~\ref{lemma.x.nilpotent}
  there is a $u_0$ such that $(y_0(y_0-1))^{u_0} = 0$. Define
  $u_{n+1}=\lceil\frac{u_n}{2}\rceil$. Clearly there is an $N$ such that $u_N
  = 1$. Then let $y=y_N$. Clearly $y$ is a polynomial in $x$ and $y(y-1) = 0$
  so that $y$ is idempotent. Finally if $\phi(y_n) = e$ then
  \begin{equation}
    \phi(y_{n+1}) = \phi(1-(1-y_n^2)^2) = 1 -(1-e^2)^2 = 1-(1-e^2) = e\,,
  \end{equation}
  so that $\phi(y) = e$ by induction.

  Note that the nilpotency order $u_0$ is smaller than the dimension of the
  algebra. Hence the choice $N=\lceil\log_2(\dim(A))\rceil$ is correct for
  all $x\in A$.
\end{proof}

In practical implementations, the given bound is much too large. A better method
is to test during the iteration of $y_{n+1}:= 1-(1-y_n^2)^2$ whether $y_n^2=y_n$ and
to stop if it holds.

For a given $\JJ$-trivial monoid, we choose $P$ according to the size of the
monoid and therefore, for a given $x$, denote by $P(x)$ the corresponding
idempotent.

Recall that in the semi-simple quotient, Equation~\eqref{equation.g} defines a
maximal decomposition of the identity $1 = \sum_{e\in\idempMon} g_e$ using the
M\"obius function. Furthermore, $g_e$ is uni-triangular and moreover by Lemma
\ref{lemma.j.idemp} $g_e = eg_e=g_ee$.
\medskip

Now pick an enumeration (that is a total ordering) of the set of
idempotents:
\begin{equation}
  \idempMon = \{e_1, e_2, \dots, e_k\} \qquad \text{and} \qquad g_i := g_{e_i}\,.
\end{equation}
Then define recursively
\begin{gather}
  f_1 := P(g_1),\quad  f_2 := P\left((1-f_1)g_2(1-f_1)\right),\quad \dots \\
  \text{and for $i>1$,}\quad 
  f_i := P\left((1-\sum_{j<i} f_j)g_i(1-\sum_{j<i} f_j)\right).
\end{gather}

We are now in position to prove Theorem~\ref{theorem.idempotents.lifting}:
\begin{proposition}
\label{proposition.unitriangular}
  The $f_i$ defined above form a uni-triangular decomposition of the identity
  compatible with the semi-simple quotient.
\end{proposition}
\begin{proof}
  First it is clear that the $f_i$ are pairwise orthogonal idempotents.
  Indeed, since $P$ has no constant term one can write $f_i$ as
  \begin{equation}
    f_i = (1-\sum_{j<i} f_j)U\,.
  \end{equation}
  Now, assuming that the $(f_j)_{j<i}$ are orthogonal, the product $f_k f_i$ with $k<i$
  must vanish since $f_k(1-\sum_{j<i} f_j) = f_k - f_k = 0$. Therefore one obtains
  by induction that for all $j< i$, $f_j f_i = 0$. The same reasoning shows
  that $f_i f_j = 0$ with $j<i$.

  Next, assuming that $\phi(f_j) = g_j$ holds for all $j<i$, one has
  \begin{equation}
    \phi\left((1-\sum_{j<i} f_j)g_i(1-\sum_{j<i} f_j)\right) =
    (1-\sum_{j<i} g_j)g_i(1-\sum_{j<i} g_j) = g_i\,.
  \end{equation}
  As a consequence $\phi(f_i) = \phi(P(g_i)) = P(\phi(g_i)) = g_i$. So that again
  by induction $\phi(f_i) = g_i$ holds for all $i$. Now $\phi(\sum_i f_i) = \sum_i g_i
  = 1$. As a consequence $1 - \sum_i f_i$ lies in the radical and must
  therefore be nilpotent. But, by orthogonality of the $f_i$ it must be
  idempotent as well:
  \begin{multline}
    (1 - \sum_i f_i)^2 = 1 - 2\sum_i f_i + (\sum_i f_i)^2
    = 1 - 2\sum_i f_i + \sum_i f_i^2 =\\
    1 - 2\sum_i f_i + \sum_i f_i = 1 - \sum_i f_i\,.
  \end{multline}
  The only possibility is that $1 - \sum_i f_i = 0$.

  It remains to show triangularity. Since the polynomial $P$ has no
  constant term $f_i$ is of the form $f_i = Ag_iB$ for $A,B\in
  \K\tMonoid$. One can therefore write $f_i = Ae_ig_iB$. By definition of the
  $\JJ$-order, any element of the monoid appearing with a nonzero coefficient
  in $f_i$ must be smaller than or equal to $e_i$. Finally, using $\phi$ one shows that the
  coefficient of $e_i$ in $f_i$ must be $1$ because the coefficient of $e_i$
  in $g_i$ is $1$ and that if $x <_\JJ e_i$ then $\phi(x) = x^{\omega}<_\JJ e_i$.
\end{proof}

%%%%%%%%%%%%%%%%%%%%%%%%%%%%%%%%%%%%%%%%%%%%%%%%%%%%%%%%%%%%%%%%%%%%%%%%%%%%%%
\subsection{The Cartan matrix and indecomposable projective modules}
\label{ss.cartan}

In this subsection, we give a combinatorial description of the Cartan
invariants of a $\JJ$-trivial monoid as well as its left and right
indecomposable projective modules. The main ingredient is the notion of
$\operatorname{lfix}$ and $\operatorname{rfix}$ which generalize left and
right descent classes in $H_0(W)$.

\begin{proposition}
  \label{proposition.aut}
  For any $x\in \tMonoid$, the set 
  \begin{equation}
    \raut{x}:=\{u\in \tMonoid \suchthat xu=x\}
  \end{equation}
  is a submonoid of $\tMonoid$. Moreover, its $\JJ$-smallest element
  $\rfix{x}$ is the unique idempotent such that
  \begin{equation}
    \raut{x} = \{u\in \tMonoid \suchthat \rfix{x} \leq_\JJ u\}\,.
  \end{equation}
  The same holds for the left: there exists a unique idempotent $\lfix{x}$
  such that
  \begin{equation}
    \label{eq.laut}
    \laut{x} := \{u\in \tMonoid \suchthat ux=x\}
              = \{u\in \tMonoid \suchthat \lfix{x} \leq_\JJ u\}\,.
  \end{equation}
\end{proposition}
\begin{proof}
  The reasoning is clearly the same on the left and on the right. We write
  the right one. The fact that $\raut{x}$ is a submonoid is clear. Pick a
  random order on $\raut{x}$ and define
  \begin{equation}
    r := \left(\prod_{u\in\raut{x}} u\right)^\omega\,.
  \end{equation}
  Clearly, $r$ is an idempotent which belongs to $\raut{x}$. Moreover, by the
  definition of $r$, for any $u\in\raut{x}$, the inequality $r\leq_\JJ u$ holds.
  Hence $\rfix{x} = r$ exists. Finally it is unique by antisymmetry of $\leq_\JJ$
  (since $\tMonoid$ is $\JJ$-trivial).
\end{proof}
Note that, by Lemma~\ref{lemma.j.idemp}, 
\begin{align}
  \label{eq.rfix}
  \rfix{x} &= \min \{e\in \idempMon \suchthat xe=x\}\,, \\
  \label{eq.lfix}
  \lfix{x} & = \min \{e\in \idempMon \suchthat ex=x\}\,,
\end{align}
the $\min$ being taken for the $\JJ$-order. These are called the
\textit{right} and \textit{left symbol} of $x$, respectively.

We recover some classical properties of descents:
\begin{proposition}
  \label{proposition.lfix.decreasing}
  $\operatorname{lfix}$ is decreasing for the $\RR$-order. Similarly,
  $\operatorname{rfix}$ is decreasing for the $\LL$-order.
\end{proposition}
\begin{proof}
  By definition, $\lfix{a}ab = ab$, so that $\lfix{a}\in\laut{ab}$. One
  concludes that $\lfix{ab}\le_\RR\lfix{a}$.
\end{proof}
\bigskip

\subsubsection{The Cartan matrix}
We now can state the key technical lemma toward the construction of
the Cartan matrix and indecomposable projective modules.
\begin{lemma}
  For any $x\in \tMonoid$, the tuple $(\lfix{x}, \rfix{x})$ is the unique
  tuple $(i,j)$ in $\idempMon \times \idempMon$ such that $f_i x$ and $x f_j$ have a nonzero
  coefficient on $x$.
\end{lemma}
\begin{proof}
  By Proposition~\ref{proposition.simple}, for any $y\in\K\tMonoid$, the
  coefficient of $x$ in $xy$ is the same as the coefficient of $\epsilon_x$ in
  $\epsilon_xy$. Now since $S_x$ is a simple module, the action of $y$ on it
  is the same as the action of $\phi(y)$. As a consequence, $\epsilon_x
  f_{\rfix{x}} = \epsilon_x g_{\rfix{x}}$. Now $\epsilon_x \rfix{x} =
  \epsilon_x$, and $\epsilon_x e=0$ for any $e<_\JJ\rfix{x}$, so that
  $\epsilon_x g_{\rfix{x}} = \epsilon_x$ and $\epsilon_x f_{\rfix{x}} =
  \epsilon_x$.

  It remains to prove the unicity of $f_j$. We need to prove that for any
  $e\neq\rfix{x}$, the coefficient of $x$ in $x f_e$ is zero. Since this
  coefficient is equal to the coefficient of $\epsilon_x$ in $\epsilon_x f_e$
  it must be zero because $\epsilon_x f_e = \epsilon_x f_{\rfix{x}} f_e =
  \epsilon_x 0 = 0$
  by the orthogonality of the $f_i$.
\end{proof}

During the proof, we have seen that the coefficient is actually $1$:
\begin{corollary}\label{corollary.bx.triang}
  For any $x\in \tMonoid$, we denote $b_x := f_{\lfix{x}} x f_{\rfix{x}}$. Then,
  \begin{equation}
    b_x = x + \sum_{y <_\JJ x} c_y y\,,
  \end{equation}
  with $c_y\in \K$. Consequently, $(b_x)_{x\in \tMonoid}$ is a basis for $\K\tMonoid$.
\end{corollary}

\begin{theorem}
  The Cartan matrix of $\K\tMonoid$ defined by $c_{i,j} :=
  \dim(f_i \K\tMonoid f_j)$ for $i,j\in \idempMon$ is given by $c_{i,j} =
  |C_{i,j}|$, where
  \begin{equation}
    C_{i,j} := \{x\in \tMonoid \suchthat i=\lfix{x}\text{ and } j=\rfix{x} \}\,.
  \end{equation}
\end{theorem}
\begin{proof}
  For any $i,j\in \idempMon$ and $x\in C_{i,j}$, it is clear that $b_x$
  belongs to $f_i\K\tMonoid f_j$. Now because $(b_x)_{x\in \tMonoid}$ is a
  basis of $\K\tMonoid$ and since $\K\tMonoid =\bigoplus_{i,j\in
    \idempMon} f_i\K\tMonoid f_j$, it must be true that $(b_x)_{x\in
    C_{i,j}}$ is a basis for $f_i\K\tMonoid f_j$.
\end{proof}

\begin{example}[Representation theory of $H_0(W)$, continued]
  Recall that the left and right descent sets and content of $w\in
  W$ can be respectively defined by:
  \begin{eqnarray*}
    D_L(w) &=& \{ i\in I \mid \ell(s_iw) < \ell(w) \}, \\
    D_R(w) &=& \{ i\in I \mid \ell(w s_i) < \ell(w) \}, \\
    \cont(w) &=& \{i \in I \mid \text{$s_i$ appears in some reduced word for $w$} \},
  \end{eqnarray*}
  and that the above conditions on $s_iw$ and $ws_i$ are respectively
  equivalent to $\pi_i \pi_w =\pi_w$ and $\pi_w\pi_i = \pi_w$.
  Furthermore, writing $w_J$ for the longest element of the parabolic
  subgroup $W_J$, so that $\pi_J=\pi_{w_J}$, one has
  $\cont(\longest_J)=D_L(w_J)$, or equivalently
  $\cont(\longest_J)=D_R(w_J)$. Then, for any $w\in W$, we have
  $\pi_w^\omega=\longest_{\cont(w)}$,
  $\lfix{\pi_w}=\longest_{D_L(w)}$, and
  $\rfix{\pi_w}=\longest_{D_R(w)}$.

  Thus, the entry $a_{J,K}$ of the Cartan matrix is given by the
  number of elements $w\in W$ having those left and right descent
  sets.
\end{example}

\bigskip
\subsubsection{Projective modules}

By the same reasoning we have the following corollary:
\begin{corollary}
  The family $\{b_x \suchthat \lfix{x} = e\}$ is a basis for the right
  projective module associated to $S_e$.
\end{corollary}
Actually one can be more precise: the projective modules are
combinatorial.
\begin{theorem}
  \label{theorem.projective_modules}
  For any idempotent $e$ denote by $R(e) = eM$,
  \begin{equation*}
    R_=(e) = \{x\in eM\suchthat \lfix{x} = e\}
    \quad\text{and}\quad 
    R_<(e) = \{x\in eM\suchthat \lfix{x} <_\RR e\} \, .
  \end{equation*}
  Then, the projective module $P_e$ associated to $S_e$ is isomorphic
  to $\K R(e)/\K R_<(e)$. In particular, the projective module $P_e$
  is combinatorial: taking as basis the image of $R_=(e)$ in the
  quotient, the action of $m\in\tMonoid$ on $x\in R_=(e)$ is given by:
  \begin{equation}
    x \cdot m =
    \begin{cases}
      xm &\text{if $\lfix{xm} = e$,}\\
      0  & \text{otherwise}.
    \end{cases}
  \end{equation}
\end{theorem}
\begin{proof}
  By Proposition~\ref{proposition.lfix.decreasing}, $R(e)$ and
  $R_<(e)$ are two ideals in the monoid, so that $A := \K R(e)/\K
  R_<(e)$ is a right $\tMonoid$-module. In order to show that $A$ is
  isomorphic to $P_e$, we first show that $A/\rad A$ is isomorphic to
  $S_e$ and then use projectivity and dimension counting to conclude
  the statement.

  We claim that
  \begin{equation}
    \label{equation.radA}
    \K (R_=(e) \backslash \{e\})\subseteq \rad A\,.
  \end{equation}
  Take indeed $x\in R_=(e) \backslash \{e\}$. Then, $x^\omega$ is
  in $\K R_<(e)$ since $\lfix{x^\omega}=x^\omega \leq_\RR x <_\RR e$. If
  follows that, in $A$, $x=x-x^\omega=e(x-x^\omega)$ which, by
  Proposition~\ref{proposition.basis.radical}, is in $\rad A$.

  Since $\rad A\subset A$, the inclusion in \eqref{equation.radA} is
  in fact an equality, and $A / \rad A$ is isomorphic to $S_e$. Then,
  by the definition of projectivity, any isomorphism from $S_e =
  P_e/\rad P_e$ to $A/\rad A$ extends to a surjective morphism from
  $P_e$ to $A$ which, by dimension count, must be an isomorphism.
\end{proof}

\begin{figure}
  \includegraphics[width=\textwidth]{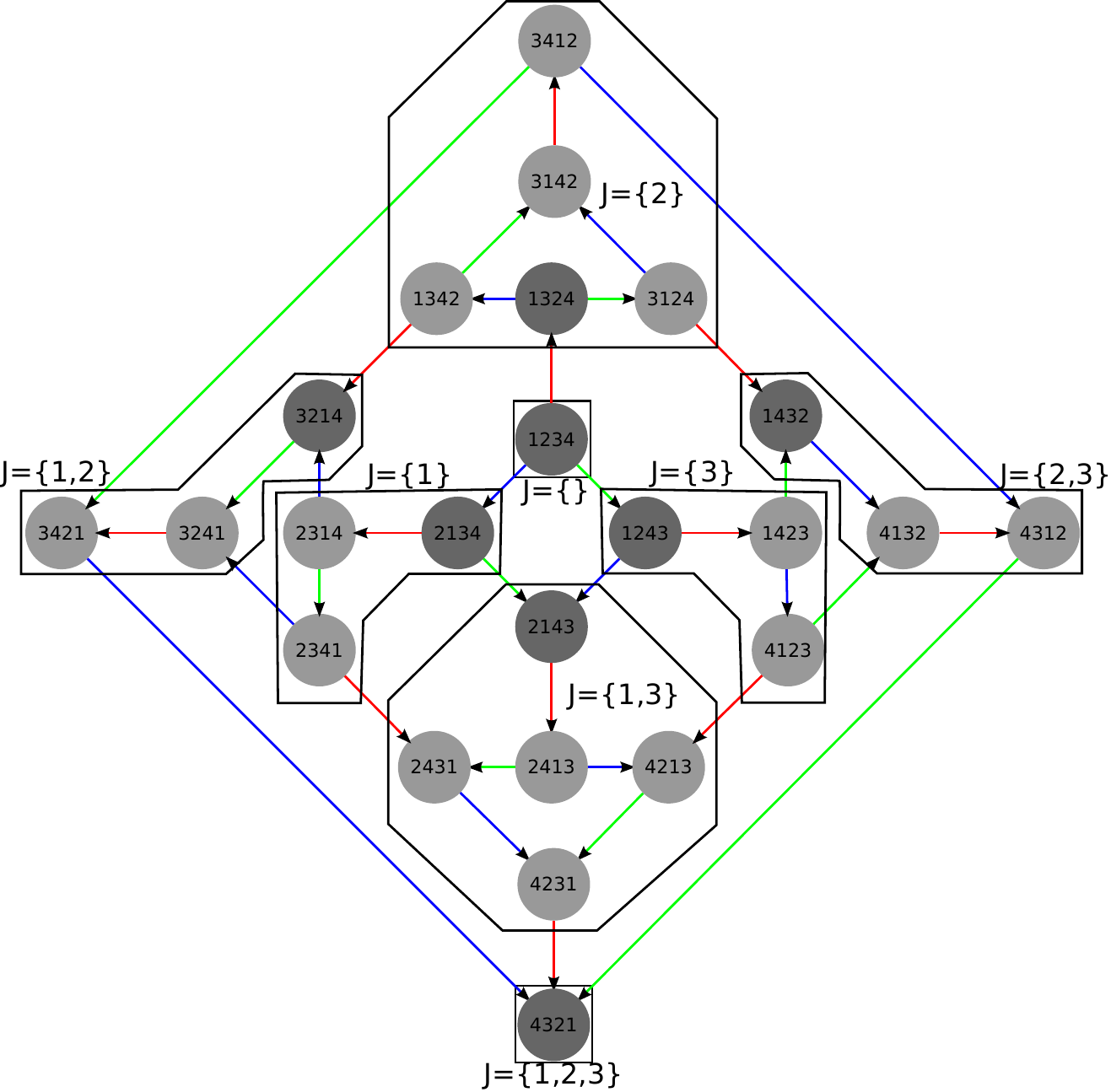}
  \caption{The decomposition of $H_0(\sg[4])$ into indecomposable
    right projective modules. This decomposition follows the partition
    of $\sg[4]$ into left descent classes, each labelled by its
    descent set $J$.  The blue, red, and green lines indicate the
    action of $\pi_1, \pi_2,$ and $\pi_3$ respectively.  The darker circles
    indicate idempotent elements of the monoid. }
  \label{h0s4.projectives}
\end{figure}

\begin{example}[Representation theory of $H_0(W)$, continued]
  \label{example.zero.hecke.projectives}
  The right projective modules of $H_0(W)$ are combinatorial, and
  described by the decomposition of the right order along left descent
  classes, as illustrated in Figure~\ref{h0s4.projectives}. Namely,
  let $P_J$ be the right projective module of $H_0(W)$ corresponding
  to the idempotent $\pi_J$. Its basis $b_w$ is indexed by the
  elements of $w$ having $J$ as left descent set. The action of
  $\pi_i$ coincides with the usual right action, except that
  $b_w.\pi_i=0$ if $w.\pi_i$ has a strictly larger left descent set than
  $w$.

  Here we reproduce Norton's construction of $P_J$~\cite{Norton.1979},
  as it is close to an explicit description of the isomorphism in the
  proof of Theorem~\ref{theorem.projective_modules}. First, notice
  that the elements $\{\pi_i^-=(1-\pi_i) \mid i\in I\}$ are idempotent
  and satisfy the same Coxeter relations as the $\pi_i$.  Thus, the
  set $\{\pi_i^-\}$ generates a monoid isomorphic to $H_0(W)$. For each
  $J\subseteq I$, let $\longest_J^-$ be the longest element in the
  parabolic submonoid associated to $J$ generated by the $\pi_i^-$
  generators, and $\longest_J^+=\longest_J$. For each subset
  $J\subseteq I$, let $\hat{J}=I\setminus J$. Define
  $f_J=\longest_{\hat{J}}^-\longest_J^+$. Then, $f_J \pi_w=0$ if
  $J\subset D_L(w)$. It follows that the right module $f_J H_0(W)$ is
  isomorphic to $P_J$ and its basis $\{ f_J\pi_w\suchthat D_L(w)=J\}$
  realizes the combinatorial module of $P_J$.

  One should notice that the elements
  $\longest_{\hat{J}}^-\longest_J^+$ are, in general, neither
  idempotent nor orthogonal. Furthermore,
  $\longest_{\hat{J}}^-\longest_J^+H_0(W)$ is not a submodule of
  $\pi_J H_0(W)$ as in the proof of
  Theorem~\ref{theorem.projective_modules}.

  The description of left projective modules is symmetric.
\end{example}

% TODO: work on this, and either solve it or rephrase the statement to
% something more comprehensible in the final version
% \begin{problem}
%   Generalize Norton's construction to build, for any $\JJ$-trivial
%   monoid $\tMonoid$, a basis of $\K\tMonoid$ which realizes
%   simultaneously, by restriction, the combinatorial model of each
%   $P_e$. Under which condition (self-injectivity of $\K M$?) is it
%   possible to further impose that the above realization of each $P_e$
%   be a submodule of $\K e\tMonoid$?
% \end{problem}

%%%%%%%%%%%%%%%%%%%%%%%%%%%%%%%%%%%%%%%%%%%%%%%%%%%%%%%%%%%%%%%%%%%%%%%%%%%%%%
\subsection{Factorizations}
\label{ss.factorizations}

It is well-known that the notion of factorization $x=uv$ and of
irreducibility play an important role in the study of $\JJ$-trivial monoids
$\tMonoid$. For example, the irreducible elements of $\tMonoid$ form the
unique minimal generating set of $\tMonoid$~\cite{Doyen.1984,Doyen.1991}. In this section, we further
refine these notions, in order to obtain in the next section a combinatorial
description of the quiver of the algebra of $\tMonoid$.

\bigskip

Let $x$ be an element of $\tMonoid$, and $e:=\lfix{x}$ and
$f:=\rfix{x}$. By Proposition~\ref{proposition.aut}, if $x=uv$ is a
factorization of $x$ such that $eu=e$ (or equivalently $e \leq_\JJ
u$), then $u\in\laut{x}$, that is $ux=x$. Similarly on the right side,
$vf = f$ implies that $xv = x$. The existence of such trivial
factorizations for any element of $\tMonoid$, beyond the usual
$x=1x=x1$, motivate the introduction of refinements of the usual
notion of proper factorizations.
\begin{definition}
  Take $x\in \tMonoid$, and let $e:=\lfix{x}$ and $f:=\rfix{x}$. A
  factorization $x=uv$ is
  \begin{itemize}
  \item \emph{proper} if $u \neq x$ and $v \neq x$;
  \item \emph{non-trivial} if $eu \neq e$ and $vf \neq f$ (or equivalently $e
    \not\le_\JJ u$ and $f \not \le_\JJ v$, or $u\notin \laut{x}$ and
    $v\notin \raut{x}$);
  \item \emph{compatible} if $u$ and $v$ are non-idempotent and
    \begin{equation*}
      \lfix{u} = e\,,\quad \rfix{v} = f\,\quad\text{and}\quad
      \rfix{u} =\lfix{v}\,.
  \end{equation*}
  \end{itemize}
\end{definition}
\begin{example}
Among the factorizations of $\pi_2\pi_1\pi_3\pi_2$ in $H_0(\sg[4])$, the
following are non-proper and trivial:
\begin{equation*}
(\id, \pi_2\pi_1\pi_3\pi_2)\quad
(\pi_2, \pi_2\pi_1\pi_3\pi_2)\quad
(\pi_2\pi_1\pi_3\pi_2, \id)\quad
(\pi_2\pi_1\pi_3\pi_2, \pi_2)\,.
\end{equation*}
The two following factorizations are proper and trivial:
\begin{equation*}
(\pi_2, \pi_1\pi_3\pi_2)\quad
(\pi_2\pi_1\pi_3, \pi_2)\,.
\end{equation*}
Here are the non-trivial and incompatible factorizations:
\begin{equation*}
\begin{array}{ccc}
(\pi_2\pi_1, \pi_3\pi_2) &
(\pi_2\pi_3, \pi_1\pi_2) &
(\pi_2\pi_1, \pi_1\pi_3\pi_2) \\
(\pi_2\pi_3, \pi_1\pi_3\pi_2) &
(\pi_2\pi_1\pi_3, \pi_1\pi_2) &
(\pi_2\pi_1\pi_3, \pi_3\pi_2)\,.
\end{array}
\end{equation*}
The only non-trivial and compatible factorization is:
\begin{equation*}
(\pi_2\pi_1\pi_3, \pi_1\pi_3\pi_2)\,.
\end{equation*}
\end{example}

\begin{lemma}
  \label{lemma.factorization.nontrivial.proper}
  Any non-trivial factorization is also proper.
\end{lemma}
\begin{proof}
  Indeed by contraposition, if $x=xv$ then $v\in\raut{x}$ and therefore
  $\rfix{x} \le_\JJ v$. The case $x=vx$ can be proved similarly.
\end{proof}

\begin{lemma}
  If $x$ is an idempotent, $x$ admits only trivial factorizations.
\end{lemma}
\begin{proof}
  Indeed if $x$ is idempotent then $x=\rfix{x}=\lfix{x}$. Then from $x=uv$,
  one obtains that $x=xuv$. Therefore $x \le_\JJ xu \le_\JJ x$ and therefore $x=xu$.
\end{proof}

\begin{lemma}
  Any compatible factorization is non-trivial.
\end{lemma}
\begin{proof}
  Let $x=uv$ be a compatible factorization. Then $\lfix{u} = e$ implies that
  $eu=u$. Since $u$ is not idempotent it cannot be equal to $e$ so that $eu
  \neq e$. The same holds on the other side.
\end{proof}

We order the factorizations of $x$ by the product $\JJ$-order: Suppose that
$x=uv=u'v'$. Then we write $(u, v) \leq_\JJ (u', v')$ if and only if $u\leq_\JJ
u'$ and $v\leq_\JJ v'$.
\begin{lemma}
  If $x = uv$ is a non-trivial factorization which is minimal for the product
  $\JJ$-order, then it is compatible.
\end{lemma}
\begin{proof}
  Let $x = uv$ be a minimal non-trivial factorization. Then $(eu, vf)$ with
  $e=\lfix{x}$ and $f=\rfix{x}$ is a
  factorization of $x$ which is also clearly non-trivial. By minimality we
  must have that $u=eu$ and $v=vf$. On the other hand, $\lfix{u}x = \lfix{u}uv
  = uv = x$, so that $e=\lfix{x}\leq_\JJ\lfix{u}$ and therefore
  $e=\lfix{u}$. This in turn implies that $u$ is non-idempotent since it is
  different from its left fix. The same holds on the right side.

  It remains to show that $\rfix{u} =\lfix{v}$.  If $g$ is an idempotent such
  that $ug=u$, then $x=u(gv)$ is a non-trivial factorization, because $gvf\leq_\JJ
  vf <_\JJ f$ so that $gvf \neq f$. Therefore by minimality, $gv=v$. By symmetry
  $ug=u$ is equivalent to $gv=v$.
\end{proof}

Putting together these two last lemmas we obtain:
\begin{proposition}
  Take $x\in \tMonoid$. Then the following are equivalent:
  \begin{enumerate}
  \item $x$ admits a non-trivial factorization;
  \item $x$ admits a compatible factorization.
  \end{enumerate}
\end{proposition}

\begin{definition}
  An element is called \emph{irreducible} if it admits no proper factorization.
  The set of all irreducible elements of a monoid $\tMonoid$ is denoted by
  $\irr$.

  An element is called \emph{c-irreducible} if it admits
  no non-trivial factorization. The set of all c-irreducible elements of a
  monoid $\tMonoid$ is denoted by $\cirr$.

  We also denote by $\quiv$ the set of c-irreducible non-idempotent elements. 
\end{definition}

\begin{remark}
  \label{remark.factorization.irr.cirr}
  By Lemma~\ref{lemma.factorization.nontrivial.proper}, $\irr\subseteq
  \cirr$. In particular $\cirr$ generates $\tMonoid$.
\end{remark}

\bigskip
%%%%%%%%%%%%%%%%%%%%%%%%%%%%%%%%%%%%%%%%%%%%%%%%%%%%%%%%%%%%%%%%%%%%%%%%%%%%%%
\subsection{The Ext-quiver}
\label{ss.quiver}

The goal of this section is to give a combinatorial description of the
quiver of the algebra of a $\JJ$-trivial monoid. We start by recalling
some well-known facts about algebras and quivers.\medskip

Recall that a quiver $Q$ is a directed graph where loops and multiple
arrows between two vertices are allowed. The path algebra $\K Q$ of
$Q$ is defined as follows. A path in $Q$ is a sequence of arrows $a_n
a_{n-1} \cdots a_3 a_2 a_1$ such that the head of $a_{i+1}$ is equal to
the tail of $a_i$. The product of the path algebra is defined by
concatenating paths if tail and head matches and by zero
otherwise. Let $F$ denote the ideal in $\K Q$ generated by the arrows of
$Q$. An ideal $I \subseteq \K Q$ is said to be \emph{admissible} if
there exists an integer $m\geq2$ such that $F^m \subseteq I \subseteq
F^2$.  An algebra is called \emph{split basic} if and only if all the
simple $A$-modules are one-dimensional. The relevance of quivers comes
from the following theorem:
\begin{theorem}[See e.g.~\cite{Auslander.Reiten.Smaloe.1997}]
  For any finite-dimensional split basic algebra $A$, there is a
  unique quiver $Q$ such that $A$ is isomorphic to $\K Q/I$ for some
  admissible ideal $I$.
\end{theorem}
In other words, the quiver $Q$ can be seen as a first order
approximation of the algebra $A$. Note however that the ideal $I$ is
not necessarily unique.\medskip

The quiver of a split basic $\K$-algebra $A$ can be computed as
follows: Let $\{f_i\mid i\in E\}$ be a complete system of primitive
orthogonal idempotents. There is one vertex $v_i$ in $Q$ for each
$i\in E$. If $i,j \in E$, then the number of arrows in $Q$ from $v_i$
to $v_j$ is $\dim f_i\big(\rad A/\mathord{\rad^2 A}\big)f_j$. This
construction does not depend on the chosen system of idempotents.

\begin{theorem}
  Let $\tMonoid$ be a $\JJ$-trivial monoid. The quiver of the algebra
  of $\tMonoid$ is the following:
  \begin{itemize}
  \item There is one vertex $v_e$ for each idempotent $e\in\idempMon$.
  \item There is an arrow from $v_{\lfix{x}}$ to $v_{\rfix{x}}$ for
    every c-irreducible element $x\in\quiv$.
  \end{itemize}
\end{theorem}

This theorem follows from Corollary~\ref{corollary.quiver_arrow} below.

\begin{lemma}
  \label{lemma.factorization_ex}
  Let $x\in \quiv$ and set $e=\lfix{x}$ and $f=\rfix{x}$.
  Recall that, by definition, whenever $x = uv$, then either $eu = e$ or $vf = f$. Then,
  \begin{equation}
    [x,e[_\RR = \{ u \in \tMonoid \suchthat eu = u \ne e \text{ and } uf = x \}.
  \end{equation}
\end{lemma}
\begin{proof}
  Obviously, $\{ u \in \tMonoid \suchthat eu = u \ne e \text{ and } uf = x \} \subseteq [x,e[_\RR$.
  Now take $u\in [x,e[_\RR$. Then, $u=ea$ for some $a\in M$ and hence $eu=eea =ea= u\ne e$.
  Furthermore, we can choose $v$ such that $x = uv$ with $vf = v$. Since $x$ admits
  no non-trivial factorization, we must have $v=f$.
\end{proof}

\begin{proposition}
  Take $x\in\quiv$ and let $e:=\lfix{x}$ and $f:=\rfix{x}$. Then, there
  exists a combinatorial module $M_x$ with basis $\epsilon=\epsilon_x,
  \xi=\xi_x$ and action given by
  \begin{align}
    \epsilon \cdot m &:=
    \begin{cases}
      \epsilon & \text{if $m\in [e,1]_\RR$}\\
      \xi      & \text{if $m\in [x,1]_\RR \setminus [e,1]_\RR$}\\
      0        & \text{otherwise,}
    \end{cases}
    \qquad \text{and}\\
    \xi \cdot m &:=
    \begin{cases}
      \xi      & \text{if $m\in [f,1]_\RR$}\\
      0        & \text{otherwise.}
    \end{cases}
  \end{align}
  This module of dimension $2$ is indecomposable, with composition factors
  given by $[e] + [f]$.
\end{proposition}

\begin{proof}
  We give a concrete realization of $M_x$. Let $I_x := e\tMonoid \setminus
  [x,e]_\RR$. This is a right ideal, and we endow the interval $[x,e]_\RR$
  with the quotient structure of $e\tMonoid / I_x$.  The second step is to
  further quotient this module by identifying all elements in $[x,e[_\RR$. Namely, define
  \begin{equation}
    \begin{split}
      \Theta: [x,e]_\RR &\rightarrow M_x\\
      e       &\mapsto \epsilon\\
      u       &\mapsto \xi \quad \text{for $u\in [x,e[_\RR.$}
    \end{split}
  \end{equation}
  It remains to prove that this map is compatible with the right action of
  $\tMonoid$. This boils down to checking that, for $u\in [x,e[_\RR$ and $y\in
  \tMonoid$:
  \begin{equation}
    uy \in [x,e[_\RR \quad \Longleftrightarrow \quad y \in [f,1]_\RR \; .
  \end{equation}
  Recall that, by Lemma~\ref{lemma.factorization_ex}, $uf=x$. Hence,
  for $y\in [f,1]_\RR$, $uy \ge_\RR uf = x$. Also, since $u\in [x,e[_\RR$ we have that
  $uy<_\RR e$. 
  Now take $y$ such that $uy\in [x,e[_\RR$, and let $v=yf$. Then $uv = uyf = x$, while
  $v=vf$. Therefore, since $x$ is c-irreducible, $v=f$.
\end{proof}

\begin{corollary}\label{corollary.irr.free}
  The family $(x-x^\omega)_{x\in\quiv}$ is free modulo $\rad^2\K\tMonoid$.
\end{corollary}
\begin{proof}
  We use a triangularity argument: If some $y\in\K\tMonoid$ lies in
  $\rad^2\K\tMonoid$ it must act by zero on all modules without square
  radical. In particular it must act by zero on all $2$-dimensional modules.
  Suppose that
  \begin{equation}
    \sum_{x\in\quiv} c_x(x-x^\omega)
  \end{equation}
  with $c_x\in\K$ acts by zero on all the previously constructed modules
  $M_x$. Suppose that some $c_x$ is nonzero and choose such an $x_0$ maximal
  in $\JJ$-order. Consider the module $M := M_{x_0}$. 
  Since $x_0\in Q(M)$, $x_0$ is not idempotent so that 
  $x_0^\omega \leq_\JJ x_0 <_\JJ \rfix{x_0}$.
  As a consequence
  \begin{equation}
    \epsilon_{x_0}\cdot x_0 = \xi_{x_0}
    \qquad\text{and}\qquad
    \epsilon_{x_0}\cdot x_0^{\omega} = 0\,.
  \end{equation}
  Moreover, if $x$ is not bigger than $x_0$ in $\JJ$-order, then
  $x$ is also not bigger than $x_0$ in $\RR$-order, so that $\epsilon_{x_0}\cdot x = 0$. Therefore
  \begin{equation}
    \epsilon_{x_0} \cdot \left(\sum_{x\in\quiv} c_x(x-x^\omega)\right)
    = c_{x_0} \xi_{x_0}
  \end{equation}
  which must vanish in contradiction with the assumption.
\end{proof}

We now show that the square radical $\rad^2\K\tMonoid$ is at least as
large as the number of factorizable elements:
\begin{proposition}
  Suppose that $x=uv$ is a non-trivial factorization of $x$. Then
  \begin{equation}
    (u-u^\omega)(v-v^\omega) = x + \sum_{y<_\JJ x} c_y y
  \end{equation}
  for some scalars $c_y \in \K$.
\end{proposition}
\begin{proof}
  We need to show that $u^\omega v$ and $uv^\omega$ are both different from
  $x$. Suppose that $u^\omega v = x$. Then $u^\omega x=x$ so that $\lfix{x} \leq_\JJ u^\omega$.  
  Since $u \lfix{x} \in \laut{x}$, we have $\lfix{x} \le_\JJ u \lfix{x} \le_\JJ \lfix{x}$. 
  Thus $u \lfix{x} = \lfix{x}$ contradicting the non-triviality of the factorization $uv$.
  The same reasoning shows that $uv^\omega <_\JJ x$.
\end{proof}

\begin{corollary}
  \label{corollary.quiver}
  The family $(x-x^\omega)_{x\in\quiv}$ is a basis of
  $\rad\K\tMonoid/\rad^2\K\tMonoid$.
\end{corollary}
\begin{proof}
  By Corollary~\ref{corollary.irr.free} we know that
  $\rad\K\tMonoid/\rad^2\K\tMonoid$ is at least of dimension
  $\card(\quiv)$. We just showed that $\rad^2\K\tMonoid$ is at least of
  dimension $\card(\tMonoid) - \card(\idempMon) - \card(\quiv)$. Therefore all
  those inequalities must be equalities.
\end{proof}

We conclude by an explicit description of the arrows of the quiver as
elements of the monoid algebra.
\begin{corollary}
\label{corollary.quiver_arrow}
  For all idempotents $i,j\in\idempMon$, the family $(f_i(x-x^\omega)f_j)$
  where $x$ runs along the set of non-idempotent c-irreducible elements such
  that $\lfix{x}=i$ and $\rfix{x} = j$ is a basis for $f_i \rad \K\tMonoid f_j$
  modulo $\rad^2\K\tMonoid$.
\end{corollary}
\begin{proof}
  By Corollary~\ref{corollary.bx.triang}, one has $(f_i x f_j) = x + \sum_{y <_\JJ x} c_y
  y$. Since $x^\omega <_\JJ x$, such a triangularity must also hold for
  $(f_i(x-x^\omega)f_j)$.
\end{proof}

\begin{remark}
  By Remark~\ref{remark.factorization.irr.cirr} a $\JJ$-trivial monoid
  $\tMonoid$ is generated by (the labels of) the vertices and the
  arrows of its quiver.
\end{remark}

\begin{lemma}
  \label{lemma.quiver}
  If $x$ is in the quiver, then it is of the form $x=epf$ with $p$
  irreducible, $e = \lfix{x}$, and $f = \rfix{x}$.
  Furthermore, if $p$ is idempotent, then $x=ef$.
\end{lemma}
\begin{proof}
  Since $x=ex=xf$, one can always write $x$ as $x=eyf$. Assume that $y$
  is not irreducible, and write $y=uv$ with $u,v<_\JJ y$. Then, since
  $x$ is in the quiver, one has either $eu=e$ or $vf=f$, and
  therefore $x=euf$ or $x=evf$. Repeating the process inductively
  eventually leads to $x=epf$ with $p$ irreducible.

  Assume further that $p$ is an idempotent. Then, $x = (e p) (p f)$
  and therefore $ep=e$ or $pf=f$. In both cases, $x=ef$.
\end{proof}
\begin{corollary}
\label{corollary.jidem_quiver}
  In a $\JJ$-trivial monoid generated by idempotents, the
  quiver is given by a subset of all products $ef$ with $e$ and $f$
  idempotents such that $e$ and $f$ are respectively the left and
  right symbols of $ef$.
\end{corollary}

\subsection{Examples of Cartan matrices and quivers}
\label{ss.quiver.examples}

We now use the results of the previous sections to describe the Cartan
matrix and quiver of several monoids. Along the way, we discuss
briefly some attempts at describing the radical filtration, and
illustrate how certain properties of the monoids (quotients,
(anti)automorphisms, ...) can sometimes be exploited.

\subsubsection{Representation theory of $H_0(W)$ (continued)}

We start by recovering the description of the quiver of the $0$-Hecke
algebra of Duchamp-Hivert-Thibon~\cite{Duchamp_Hivert_Thibon.2002} in
type $A$ and of Fayers~\cite{Fayers.2005} in general type. We further
refine it by providing a natural indexing of the arrows of the
quiver by certain elements of $H_0(W)$.
\begin{proposition}
  \label{proposition.quiver.hecke}
  The quiver elements $x\in\quiv$ are exactly the products
  $x=\longest_J\longest_K$ where $J$ and $K$ are two incomparable
  subsets of $I$ such that, for any $j\in J\setminus K$ and $k\in
  K\setminus J$, the generators $\pi_j$ and $\pi_k$ do not commute.
\end{proposition}
\begin{proof}
  Recall that the idempotents of $H_0(W)$ are exactly the $\longest_J$
  for all subsets $J$ and that by Corollary~\ref{corollary.jidem_quiver}, the
  c-irreducible elements are among the products $\longest_J\longest_K$.

  First of all if $J\subseteq K$ then
  $\longest_J\longest_K=\longest_K\longest_J=\longest_K$ so that, for
  $\longest_J\longest_K$ to be c-irreducible, $J$ and $K$ have to be
  incomparable. Now suppose that there exists some $j\in J\setminus K$ and
  $k\in K\setminus J$ such that $\pi_j\pi_k=\pi_k\pi_j$. Then
  \begin{equation}
    \longest_J\longest_K=\longest_J\pi_j\pi_k\longest_K=\longest_J\pi_k\pi_j\longest_K\,.
  \end{equation}
  But since $k\notin J$, one has $\longest_J\pi_k\neq\longest_J$. Similarly,
  $\pi_j\longest_K\neq\longest_K$. This implies that $(\longest_J\pi_k,\pi_j\longest_K)$ is a
  non-trivial factorization of $\longest_J\longest_K$.

  Conversely, suppose that there exists a non-trivial factorization
  $\longest_J\longest_K = uv$. Since $\longest_Ju\neq\longest_J$,
  there must exist some $k \in K\backslash J$ such that
  $u\le_\JJ\pi_k$ (or equivalently $\pi_k$ appears in some and
  therefore any reduced word for $u$). Similarly, one can find some
  $j\in J\backslash K$ such that $v\le_\JJ\pi_j$.
  Then, for $<_\BB$ as defined in~\eqref{equation.bruhat}, that is
  reversed Bruhat order, we have
  \begin{displaymath}
    \longest_J\longest_K = \longest_Juv\longest_K \leq_\BB
    \longest_J\pi_k\pi_j\longest_K \leq_\BB \longest_J\longest_K\, ,
  \end{displaymath}
  and therefore $\longest_J\pi_k\pi_j\longest_K =
  \longest_J\longest_K$. Hence the left hand side of this equation can
  be rewritten to its right hand side using a sequence of applications
  of the relations of $H_0(W)$. Notice that using $\pi_i^2=\pi_i$ or
  any non trivial braid relation preserves the condition that there
  exists some $\pi_k$ to the left of some $\pi_j$. Hence rewriting
  $\longest_J\pi_k\pi_j\longest_K$ into $\longest_J\longest_K$
  requires using the commutation relation $\pi_k\pi_j=\pi_j\pi_k$ at
  some point, as desired.
\end{proof}

\subsubsection{About the radical filtration}

Proposition~\ref{proposition.quiver.hecke} suggests to search for a
natural indexing by elements of the monoid not only of the quiver,
but of the full \emph{Loewy filtration}.
\begin{problem}
  \label{problem.radical.filtration}
  Find some statistic $r(m)$ for $m \in \tMonoid$ such that, for any
  two idempotents $i,j$ and any integer $k$,
  \begin{multline}
    \dim f_i\big(\rad^k A /\rad^{k+1} A\big) f_j =\\
    \card\{m\in\tMonoid\mid r(m) = k,\ \lfix{m}=i,\ \rfix{m}=j\}\,.
  \end{multline}
\end{problem}

\begin{table}
\begin{equation*}
  \begin{array}{|c|c|}
    \hline
    Type & \text{Generating series} \\
    \hline
    A_1  & 2 \\
    A_2  & 2q + 4 \\
    A_3  & 6q^2 + 10q + 8 \\
    A_4  & 10q^4 + 24q^3 + 38q^2 + 32q + 16 \\
    A_5  & 14q^7 + 48q^6 + 72q^5 + 144q^4 + 172q^3 + 150q^2 + 88q + 32 \\
    \hline
    B_2  & 2q^2 + 2q + 4 \\
    B_3  & 6q^4 + 10q^3 + 14q^2 + 10q + 8 \\
    B_4  & 12q^8 + 24q^7 + 46q^6 + 60q^5 + 76q^4 + 64q^3 + 54q^2 + 32q + 16 \\
    \hline
    D_2  & 4 \\
    D_3  & 6q^2 + 10q + 8 \\
    D_4  & 6q^6 + 12q^5 + 20q^4 + 38q^3 + 62q^2 + 38q + 16 \\
    \hline
    H_3  & 6q^8 + 10q^7 + 14q^6 + 18q^5 + 22q^4 + 18q^3 + 14q^2 + 10q + 8\\
    \hline
    I_5  & 2q^3 + 2q^2 + 2q + 4 \\
    I_6  & 2q^4 + 2q^3 + 2q^2 + 2q + 4 \\
    I_n  & 2q^{n-2} + \dots + 2q^2 + 2q + 4\\
    \hline
  \end{array}
  \end{equation*}
  \caption{The generating series
    $\sum_k \dim \big(\rad^k A/\rad^{k+1}A\big)\ q^k$ for
    the $0$-Hecke algebras $A=\K H_0(W)$ of the
    small Coxeter groups.}
  \label{table.radical_filtration}
\end{table}
Such a statistic is not known for $H_0(W)$, even in type $A$. Its
expected generating series for small Coxeter group is shown in
Table~\ref{table.radical_filtration}. Note that all the coefficients
appearing there are even. This is a general fact:
\begin{proposition}
  Let $W$ be a Coxeter group and $H_0(W)$ its $0$-Hecke monoid. Then, for any
  $k$, the dimension $d^k := \dim \rad^k \K H_0(W)$ is an even number.
\end{proposition}
\begin{proof}
  This is a consequence of the involutive algebra automorphism $\theta\ :\
  \pi_i \mapsto 1-\pi_i$. This automorphism exchanges the eigenvalues 0 and 1
  for the idempotent $\pi_i$. Therefore it exchanges the projective module
  $P_J$ associated to the descent set $J$ (see
  Example~\ref{example.zero.hecke} for the definition of $P_J$) with the
  projective module $P_{\overline J}$ associated to the complementary descent
  set $\overline J = I \setminus J$. As a consequence it must exchange $\rad^k P_J$ and
  $\rad^k P_{\overline J}$ which therefore have the same dimensions. Since
  there is no self-complementary descent set, $d^k = \sum_{J\subset I} \rad^k
  P_J$ must be even.
\end{proof}

Also, as suggested by Table~\ref{table.radical_filtration},
Problem~\ref{problem.radical.filtration} admits a simple solution for
$H_0(I_n)$.
\begin{proposition}
  Let $W$ be the $n$-th dihedral group (type $I_n$) and $\K H_0(W)$ its
  $0$-Hecke algebra. 
  Define $a_k = \pi_1\pi_2\pi_1\pi_2\cdots$ and
  $b_k = \pi_2\pi_1\pi_2\pi_1\cdots$ where both words are of length
  $k$. Recall that the longest element of $H_0(W)$ is $\omega = a_n =
  b_n$. Then, for all $k>0$, the set
  \begin{equation}
    R_k := \{a_i - \omega,\ b_i - \omega\mid k<i<n\}
  \end{equation}
  is a basis for $\rad^k \K H_0(W)$. In particular, defining the
  statistic $r(w):=\len(w)-1$, one obtains that the family
  $$\{a_{k+1}-\omega,b_{k+1}-\omega\}$$ 
  for $0<k<n-1$ is a basis of $\rad^k \K
  H_0(W)/\rad^{k+1} \K H_0(W)$.
\end{proposition}
Note that if $k<n-1$ then $\omega$ belongs to $\rad^{k+1} \K H_0(W)$. One can
therefore take $\{a_{k+1},b_{k+1}\}$ as a basis.
\begin{proof}
  The case $k=1$ follows from
  Proposition~\ref{proposition.basis.radical}, and by
  Proposition~\ref{corollary.quiver} the quiver is given by
  $a_2-\omega$ and $b_2-\omega$. The other cases are then proved by
  induction, using the following relations:
  \begin{alignat*}{2}
    (a_2 - \omega)(a_j - \omega) &= a_{j+2} - \omega&\qquad
    (a_2 - \omega)(b_j - \omega) &= a_{j+1} - \omega\\
    (b_2 - \omega)(b_j - \omega) &= b_{j+2} - \omega&
    (b_2 - \omega)(a_j - \omega) &= b_{j+1} - \omega.\qedhere
  \end{alignat*}
\end{proof}

A natural approach to try to define such a statistic $r(m)$ is to use
iterated compatible factorizations. For example, one can define a new
product $\bullet$, called the \emph{compatible product} on
$\tMonoid\cup\{0\}$, as follows:
\begin{equation*}
  x \bullet y =
  \begin{cases}
    xy & \text{if $\lfix{x} = \lfix{xy}$ and
                  $\rfix{y} = \rfix{xy}$ and $\rfix{x}=\lfix{y}$,}\\
                0 & \text{otherwise.}
  \end{cases}
\end{equation*}
However this product is usually not associative. Take for example $x =
\pi_{14352}$, $y = \pi_{31254}$ and $z = \pi_{25314}$ in
$H_0(\sg[5])$. Then, $xy= \pi_{41352}$, $yz= \pi_{35214}$ and $xyz =
\pi_{45312}$. The following table shows the left and right descents of
those elements:
\begin{equation*}
  \begin{array}{r|c|c}
                    & \text{left}  & \text{right}   \\\hline
    x=\pi_{14352}    &   \{2,3\}    &    \{2,4\}     \\
    y=\pi_{31254}    &   \{2,4\}    &    \{1,4\}     \\
    z=\pi_{25314}    &   \{1,4\}    &    \{2,3\}     \\
   xy=\pi_{41352}    &   \{2,3\}    &    \{1,4\}     \\
   yz=\pi_{35214}    &   \{1,2,4\}  &    \{2,3\}     \\
  xyz=\pi_{45312}    &   \{2,3\}    &    \{2,3\}     \\
  \end{array}
\end{equation*}
Consequently $(x\bullet y)\bullet z = (xy) \bullet z = xyz$ whereas $y\bullet
z = 0$ and therefore $x\bullet (y\bullet z) = 0$. 

Due to the lack of associativity there is no immediate definition for
$r(m)$ as the ``length of the longest compatible factorization'', and our
various attempts to define this concept all failed for the $0$-Hecke
algebra in type $D_4$.  \bigskip

\subsubsection{Nondecreasing parking functions}

We present, without proof, how the description of the Cartan matrix of
$\NDPF_n$
in~\cite{Hivert_Thiery.HeckeSg.2006,Hivert_Thiery.HeckeGroup.2007}
fits within the theory, and derive its quiver from that of
$H_0(\sg[n])$.

\begin{proposition}
  The idempotents of $\NDPF_n$ are characterized by their image sets,
  and there is one such idempotent for each subset of $\{1,\dots,n\}$
  containing $1$. For $f$ an element of $\NDPF_n$, $\rfix f$ is given
  by the image set of $f$, whereas $\lfix f$ is given by the set of
  all lowest point in each fiber of $f$; furthermore, $f$ is
  completely characterized by $\lfix f$ and $\rfix f$.

  The Cartan matrix is $0,1$, with $c_{I,J}=1$ if
  $I=\{i_1<\cdots<i_k\}$ and $J=\{j_1<\cdots<j_k\}$ are two subsets of
  the same cardinality $k$ with $i_l\leq j_l$ for all $l$.
\end{proposition}

\begin{proposition}
  Let $\tMonoid$ be a $\JJ$-trivial monoid generated by
  idempotents. Suppose that $N$ is a quotient of $\tMonoid$ such that
  $\idempMon[N] = \idempMon$. Then, the quiver of $N$ is a subgraph
  of the quiver of $\tMonoid$.
\end{proposition}
Note that the hypothesis implies that $\tMonoid$ and $N$ have the same
generating set.
\begin{proof}
  It is easy to see that $\operatorname{lfix}$ and $\operatorname{rfix}$ are
  the same in $\tMonoid$ and $N$. Moreover, any compatible factorization in
  $\tMonoid$ is still a compatible factorization in $N$.
\end{proof}

As a consequence one recovers the quiver of $\NDPF_n$:
\begin{proposition}
  The quiver elements of $\NDPF_n$ are the products $\pi_{J\cup
    \{i\}}\pi_{J\cup \{i+1\}}$ where $J\subset\{1,\dots, n-1\}$ and
  $i,i+1\notin J$.
\end{proposition}
\begin{proof}
  Recall that $\NDPF_n$ is the quotient of $H_0(\sg[n])$ by the
  relation $\pi_i\pi_{i+1}\pi_i = \pi_{i+1}\pi_i$, via the quotient
  map $\pi_i\mapsto \pi_i$. For $J$ a subset of $\{1,\dots,n-1\}$,
  define accordingly $\pi_J$ in $\NDPF_n$ as the image of $\pi_J$ in
  $H_0(\sg[n])$. Specializing
  Proposition~\ref{proposition.quiver.hecke} to type $A_{n-1}$, one
  obtains that there are four types of quiver elements:
  \begin{itemize}
  \item $\pi_{J\cup \{i\}}\pi_{J\cup \{i+1\}}$ where $J\subset\{1,\dots, n-1\}$ and
  $i,i+1\notin J$,
  \item $\pi_{J\cup \{i+1\}}\pi_{J\cup \{i\}}$ where $J\subset\{1,\dots, n-1\}$ and
    $i,i+1\notin J$, 
  \item $\pi_{K\cup \{i, i+2\}}\pi_{K\cup \{i+1\}}$ where $K\subset\{1,\dots,
    n-1\}$ and $i,i+1,i+2\notin K$,
  \item $\pi_{K\cup \{i+1\}}\pi_{K\cup \{i, i+2\}}$ where $K\subset\{1,\dots,
    n-1\}$ and $i,i+1,i+2\notin K$.
  \end{itemize}
  One can easily check that the three following factorizations are non-trivial:
  \begin{itemize}
  \item $\pi_{J\cup \{i+1\}}\pi_{J\cup \{i\}} =
    (\pi_{J\cup  \{i+1\}}\pi_i,\ \pi_{i+1}\pi_{J\cup \{i\}})$,
  \item $\pi_{K\cup \{i, i+2\}}\pi_{K\cup \{i+1\}} =
    (\pi_{K\cup \{i, i+2\}}\pi_{i+1},\ \pi_{i+2}\pi_{K\cup \{i+1\}})$,
  \item $\pi_{K\cup \{i+1\}}\pi_{K\cup \{i, i+2\}} =
    (\pi_{K\cup \{i+1\}}\pi_{i},\ \pi_{i+1}\pi_{K\cup \{i, i+2\}})$.
  \end{itemize}
  Conversely, any non-trivial factorization of $\pi_{J\cup
    \{i\}}\pi_{J\cup \{i+1\}}$ in $\NDPF_n$ would have been non-trivial in the
  Hecke monoid.
\end{proof}
\bigskip

\subsubsection{The incidence algebra of a poset}

We show now that we can recover the well-known representation theory of the
incidence algebra of a partially ordered set.

Let $(P, \leq)$ be a partially ordered set. Recall that the incidence algebra of
$P$ is the algebra $\K P$ whose basis is the set of pairs $(x,y)$ of
comparable elements $x\leq y$ with the product rule
\begin{equation}\label{equation.incidence}
  (x,y) (z,t) =
  \begin{cases}
    (x,t) & \text{if $y = z$,}\\
    0     & \text{otherwise.}
  \end{cases}
\end{equation}
The incidence algebra is very close to the algebra of a monoid except that
$0$ and $1$ are missing. We therefore build a monoid by adding $0$ and $1$ artificially
and removing them at the end:
\newcommand{\Zero}{\operatorname{Zero}}
\newcommand{\One}{\operatorname{One}}
\begin{definition}
  Let $(P, \leq)$ be a partially ordered set. Let $\Zero$ and $\One$ be two
  elements not in $P$. The incidence monoid of $P$ is the monoid $M(P)$, whose
  underlying set is
  \begin{equation*}
    M(P) := \{(x,y)\in P \mid x\leq y\} \cup \{\Zero, \One\}\,,
  \end{equation*}
  with the product rule given by Equation~\ref{equation.incidence} plus $\One$
  being neutral and $\Zero$ absorbing.
\end{definition}
\begin{proposition}
  Define an order $\preceq$ on $M(P)$ by
  \begin{equation}
    (x,y) \preceq (z,t)
    \quad\text{if and only if}\quad
    x\leq z\leq t\leq y\,,
  \end{equation}
  and $\One$ and $\Zero$ being the largest and the smallest element, respectively.
  The monoid $M(P)$ is left-right ordered for $\preceq$ and thus
  $\JJ$-trivial.
\end{proposition}
\begin{proof}
  This is trivial by the product rule.
\end{proof}
One can now use all the results on $\JJ$-trivial monoids to obtain the
representation theory of $M(P)$. One gets back to $\K P$ thanks to the
following result.
\begin{proposition}
  As an algebra, $\K M(P)$ is isomorphic to
  $\K\One\ \oplus\ \K P\ \oplus\ \K \Zero$.
\end{proposition}
\begin{proof}
  In the monoid algebra $\K M(P)$, the elements $(x,x)$ are orthogonal
  idempotents. Thus $e:=\sum_{x\in P} (x,x)$ is itself an idempotent and it is
  easily seen that $\K P$ is isomorphic to $e(\K M(P))e$.
\end{proof}
One can then easily deduce the representation theory of $\K P$:
\begin{proposition}
  Let $(P, \leq)$ be a partially ordered set and $\K P$ its incidence algebra.
  Then the Cartan matrix $C = (c_{x,y})_{x,y\in P}$ of $\K P$ is indexed by
  $P$ and given by
  \begin{equation*}
    c_{x,y} =
    \begin{cases}
      1 & \text{if $x\leq y\,,$}\\
      0 & \text{otherwise.}
    \end{cases}
  \end{equation*}
  The arrows of the quiver are $x\to y$ whenever $(x,y)$ is a cover in $P$,
  that is, $x \le y$ and there is no $z$ such that $x\le z\le y$.
\end{proposition}
\begin{proof}
  Clearly $\lfix{x,y} = (x,x)$ and $\rfix{x,y} = (y,y)$. Moreover, the
  compatible factorizations of $(x,y)$ are exactly $(x,z)(z,y)$ with $x<z<y$.
\end{proof}
\bigskip

\subsubsection{Unitriangular Boolean matrices}

Next we consider the monoid of unitriangular Boolean matrices $\unitribool_n$.
\begin{remark}
  The idempotents of $\unitribool_n$ are in bijection with the posets admitting
  $1,\dots,n$ as linear extension (sequence A006455 in~\cite{Sloane}).

  Let $m\in \unitribool_n$ and $g$ be the corresponding digraph. Then
  $m^\omega$ is the transitive closure of $g$, and $\lfix g$ and
  $\rfix g$ are given respectively by the largest ``prefix'' and
  ``postfix'' of $g$ which are posets: namely, $\lfix g$ (resp. $\rfix
  g$) correspond to the subgraph of $g$ containing the edges $i\edge
  j$ (resp. $j\edge k$) of $g$ such that $i\edge k$ is in $g$
  whenever $j\edge k$ (resp. $i\edge j$) is.
\end{remark}

Figure~\ref{fig:unitribool4} displays the Cartan matrix and quiver of
$\unitribool_4$; as expected, their nodes are labelled by the 40
subposets of the chain. This figure further suggests that they are
acyclic and enjoy a certain symmetry, properties which we now prove in general.

\begin{figure}
  \centering
  \fbox{\includegraphics[width=\textwidth]{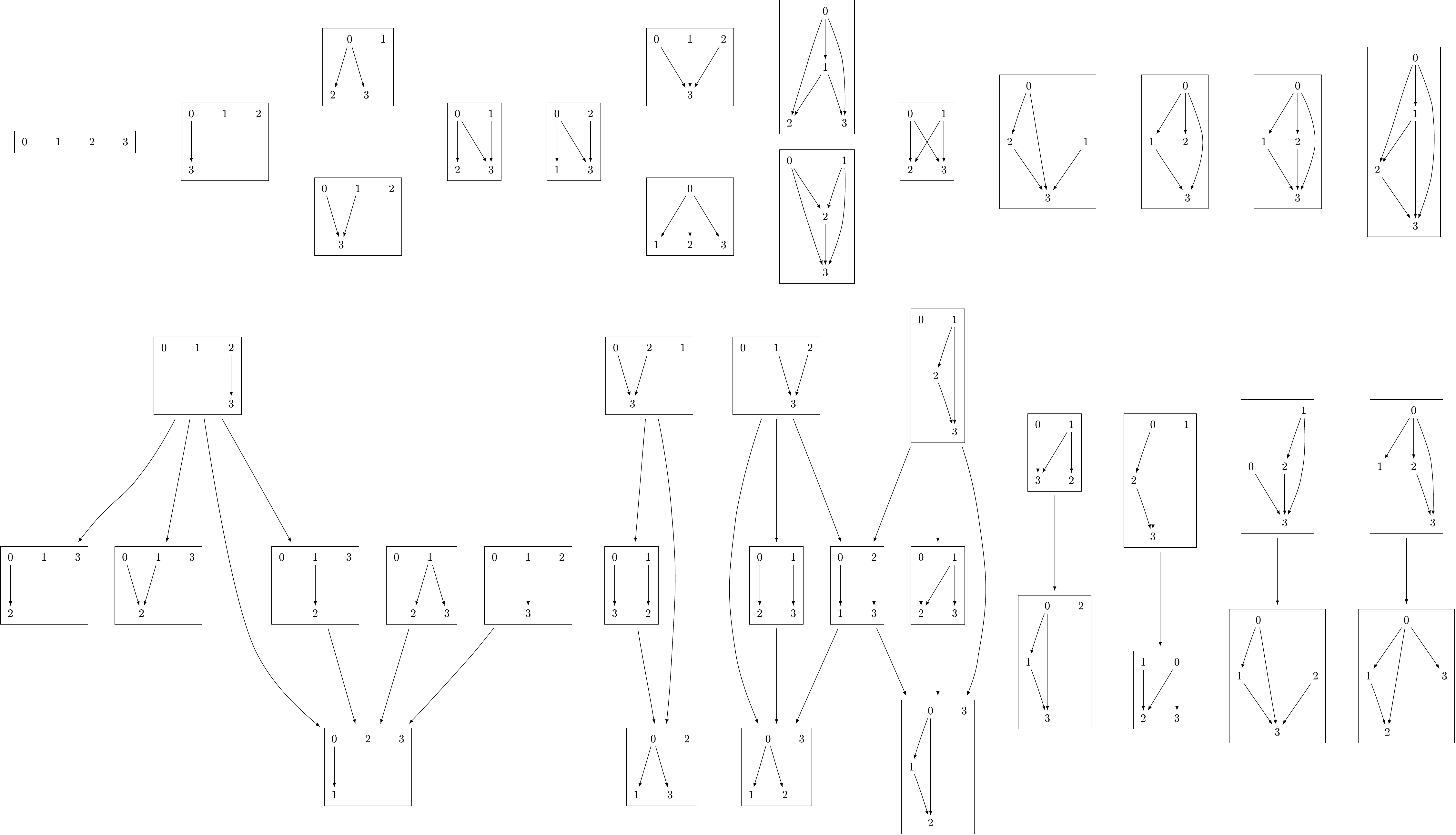}}
  \fbox{\includegraphics[width=\textwidth]{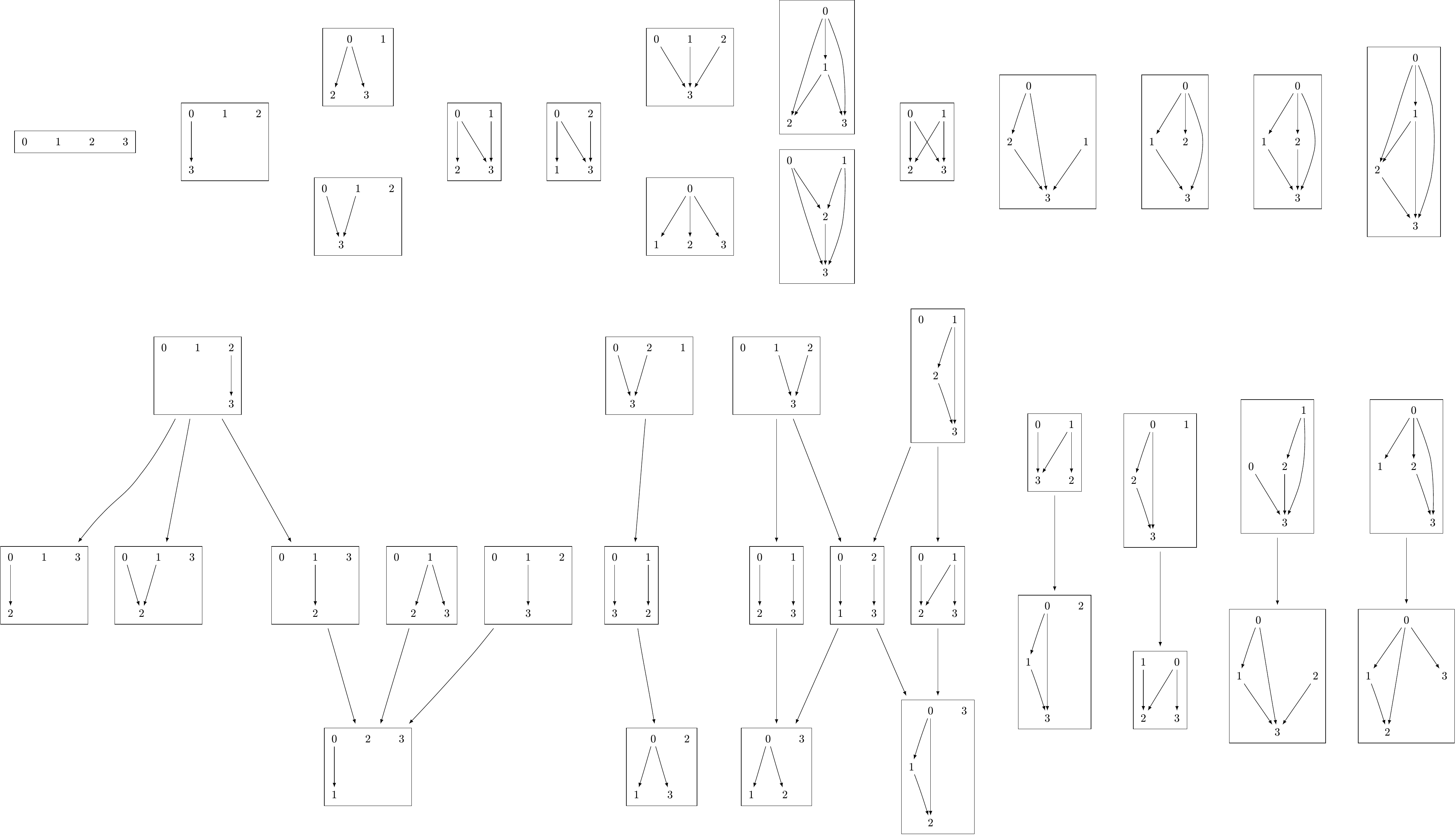}}
  \caption{On top, the Cartan matrix (drawn as a graph) and at the
    bottom the quiver of $\unitribool_4$. The edge labels have not
    been drawn for readability; for the quiver, they can be recovered
    as the product of two vertices.
    Those pictures have been produced automatically by \texttt{Sage},
    \texttt{dot2tex}, and \texttt{graphviz}, up to a manual
    reorganization of the connected components using
    \texttt{inkscape}.}
  \label{fig:unitribool4}
\end{figure}

The monoid $\unitribool_n$ admits a natural antiautomorphism $\phi$;
it maps an upper triangular Boolean matrix to its transpose along the
second diagonal or, equivalently, relabels the vertices of the
corresponding digraph by $i\mapsto n-i$ and then takes the dual.
\begin{proposition}
  The Cartan matrix of $\unitribool_n$, seen as a graph, and its quiver are
  preserved by the non-trivial antiautomorphism induced by $\phi$.
\end{proposition}

\begin{proof}
  Remark that any antiautomorphism $\phi$ flips $\operatorname{lfix}$
  and $\operatorname{rfix}$:
  \begin{displaymath}
    \lfix{\phi(x)} = \rfix x \qquad \text{and} \qquad \rfix{\phi(x)}=\lfix x\,,
  \end{displaymath}
  and that the definition of $c$-irreducible is symmetric.
\end{proof}
\begin{comment}
Note however that the monoid algebra of $\unitribool_n$ is not
Frobenius.
\begin{sageexample}
  sage: from sage.combinat.boolean_matrix_monoids import *
  sage: b = UnitriangularBooleanMatricesPosetMonoid(Posets.ChainPoset(3))
  sage: ba = b.algebra(QQ)
  sage: frob = lambda a, b: (a*b).coefficient(ba.basis().keys().list()[-1])
  sage: mat = matrix([[frob(a,b) for a in ba.basis()] for b in ba.basis()])
  sage: mat.rank()
  7
  sage: mat.kernel()
  Vector space of degree 8 and dimension 1 over Rational Field
  Basis matrix:
  [ 0  0  1  0  0  0 -1  0]
  sage: ba.basis().list()[2]
  B[[1 0 0]
  [0 1 1]
  [0 0 1]]
  sage: ba.basis().list()[-2]
  B[[1 0 1]
  [0 1 1]
  [0 0 1]]
\end{sageexample}
\end{comment}

Fix an ordering of the pairs $(i,j)$ with $i<j$ such that $(i,j)$
always comes before $(j,k)$ (for example using lexicographic order).
Compare two elements of $\unitribool_n$ lexicographically by writing
them as bit vectors along the chosen enumeration of the pairs $(i,j)$.
\begin{proposition}
  The Cartan matrix of $\unitribool_n$ is unitriangular with respect
  to the chosen order, and therefore its quiver is acyclic.
\end{proposition}
\begin{proof}
  We prove that, if $e=\lfix g$ and $f=\rfix g$, then $e\leq f$, with
  equality if and only if $g$ is idempotent.

  If $g$ is idempotent, then $e=f=g$, and we are done. Assume now that
  $g$ is not idempotent, so that $e\ne g$ and $f\ne g$. Take the
  smallest edge $j\edge k$ which is in $g$ but not in $f$. Then,
  there exists $i<j$ such that $i\edge k$ is not in $g$ but
  $i\edge j$ is. Therefore $i\edge j$ is not in $e$, whereas by
  minimality it is in $f$. Hence, $f>e$, as desired.
\end{proof}

Looking further at Figure~\ref{fig:unitribool4} suggests that the
quiver is obtained as the transitive reduction of the Cartan matrix;
we checked on computer that this property still holds for $n=5$ and $n=6$.
\begin{comment}
\begin{sageexample}
  sage: M = semigroupe.UnitriangularBooleanMatrixSemigroup(6)
  sage: f = M.quiver(edge_labels=False); g = M.cartan_matrix_as_graph(); h=g.transitive_reduction()
  sage: set(f.vertices()) == set(h.vertices())
  True
\end{sageexample}

Looking at $n=4,5$ also suggest that, if `x` is a quiver element, then
its edges are the union of the edges of the edges of its left and
right symbol. That's not an if and only if condition though:

\begin{sageexample}
  sage: def is_quiver(e,f):
  ...       ef = e * f
  ...       res = e.value+f.value;
  ...       for pos in res._nonzero_positions_by_row():
  ...           res[pos] = 1
  ...       return e != f and ef.symbol("left") == e and ef.symbol("right") == f and res == ef.value
  sage: U = UnitriangularBooleanMatricesPosetMonoid(Posets.ChainPoset(4))
  sage: set(U.quiver_elements()).issubset(set( e*f for e,f in CartesianProduct(U.idempotents(),U.idempotents()) if is_quiver(e,f)))
  True
  sage: set(U.quiver_elements()) == set( e*f for e,f in CartesianProduct(U.idempotents(),U.idempotents()) if is_quiver(e,f))
  False
\end{sageexample}

One of the two counter examples is:
\begin{sageexample}
  sage: e
  [1 0 0 0]
  [0 1 1 1]
  [0 0 1 1]
  [0 0 0 1]
  sage: f
  [1 1 1 0]
  [0 1 1 0]
  [0 0 1 0]
  [0 0 0 1]
  sage: e*f
  [1 1 1 0]
  [0 1 1 1]
  [0 0 1 1]
  [0 0 0 1]
\end{sageexample}
\end{comment}

\bigskip

\subsubsection{$\JJ$-trivial monoids built from quivers}

We conclude with a collection of examples showing in particular that any
quiver can be obtained as quiver of a finite $\JJ$-trivial monoid.

\begin{example}
  Consider a finite commutative idempotent $\JJ$-trivial monoid, that
  is a finite lattice $L$ endowed with its meet operation. Denote
  accordingly by $0$ and $1$ the bottom and top elements of
  $L$. Extend $L$ by a new generator $p$, subject to the relations
  $pep=0$ for all $e$ in $L$, to get a $\JJ$-trivial monoid $M$ with
  elements given by $L\uplus \{ e p f \suchthat e,f \in L \}$.

  Then, the quiver of $M$ is a complete digraph: its vertices are the
  elements of $L$, and between any two elements $e$ and $f$ of $L$,
  there is a single edge which is labelled by $epf$.
\end{example}

\newcommand{\efedge}{{e\!\stackrel{l}{\rightarrow}\!\! f}}
\begin{example}
  Consider any finite quiver $G=(E,Q)$, that is a digraph, possibly
  with loops, cycles, or multiple edges, and with distinct labels on
  all edges. We denote by $\efedge$ an edge in $Q$ from $e$ to $f$
  with label $l$.

  Define a monoid $M(G)$ on the set $E\uplus Q \uplus \{0,1\}$ by the
  following product rules:
  \begin{xalignat*}{2}
    e^2    &\ =\  e && \text{for all $e \in E$,}\\
    e\ \efedge  &\ =\  \efedge && \text{for all $\efedge \in Q$,}\\
    \efedge\ f &\ =\  \efedge && \text{for all $\efedge \in Q$,}
  \end{xalignat*}
  together with the usual product rule for $1$, and all other products
  being $0$.  In other words, this is the quotient of the path monoid
  $P(G)$ of $G$ (which is $\JJ$-trivial) obtained by setting $p=0$ for
  all paths $p$ of length at least two.

  Then, $M(G)$ is a $\JJ$-trivial monoid,
  and its quiver is $G$ with $0$ and $1$ added as extra isolated
  vertices. Those extra vertices can be eliminated by considering
  instead the analogous quotient of the path algebra of $G$
  (i.e. setting $0_{M(G)}=0_\K$ and $1_{M(G)}=\sum_{g\in E} g$).
\end{example}

\begin{example}
  Choose further a lattice structure $L$ on $E\cup \{0,1\}$. Define a
  $\JJ$-trivial monoid $M(G,L)$ on the set $E\uplus Q \uplus
  \{0,1\}$ by the following product rules:
  \begin{xalignat*}{2}
    e f     &\ =\  e\meet_L f && \text{for all $e,f \in E$,}\\
    \efedge\  f' &\ =\  \efedge     && \text{for all $\efedge\in Q$ and $f'\in E$ with $f\leq_L f'$,}\\
    e'\ \efedge &\ =\  \efedge      && \text{for all $\efedge\in Q$ and $e'\in E$ with $e\leq_L e'$,}
  \end{xalignat*}
  together with the usual product rule for $1$, and all other products
  being $0$. Note that the monoid $M(G)$ of the previous example is
  obtained by taking for $L$ the lattice where the vertices of $G$
  form an antichain. Then, the semi-simple quotient of $M(G,L)$ is
  $L$ and its quiver is $G$ (with $0$ and $1$ added as extra isolated
  vertices).
\end{example}
\begin{example}
  \renewcommand{\efedge}{e\,\edge f}

  We now assume that $G=(E,Q)$ is a simple quiver. Namely, there are
  no loops, and between two distinct vertices $e$ and $f$ there is at
  most one edge which we denote by $\efedge$ for short. Define a
  monoid structure $M'(G)$ on the set $E\uplus Q\uplus \{0,1\}$ by the
  following product rules:
  \begin{xalignat*}{2}
    e^2    &\ =\  e  && \text{for all $e \in E$,}\\
    e f    &\ =\  \efedge && \text{for all $\efedge \in Q$,}\\
    e\ \efedge &\ =\  \efedge && \text{for all $\efedge \in Q$,}\\
    \efedge\  f &\ =\  \efedge && \text{for all $\efedge \in Q$,}
  \end{xalignat*}
  together with the usual product rule for $1$, and all other products
  being $0$.

  Then, $M'(G)$ is a $\JJ$-trivial monoid generated by the
  idempotents in $E$ and its quiver is $G$ (with $0$ and $1$ added as
  extra isolated vertices).
\end{example}

\begin{exercise}
  Let $L$ be a lattice structure on $E\cup \{0,1\}$. Find
  compatibility conditions between $G$ and $L$ for the existence of a
  $\JJ$-trivial monoid generated by idempotents having $L$ as
  semi-simple quotient and $G$ (with $0$ and $1$ added as extra
  isolated vertices) as quiver.
\end{exercise}

%%%%%%%%%%%%%%%%%%%%%%%%%%%%%%%%%%%%%%%%%%%%%%%%%%%%%%%%%%%%%%%%%%%%%%%%%%%%%%
\subsection{Implementation and complexity}
\label{subsection.implementation_complexity}

The combinatorial description of the representation theoretical
properties of a $\JJ$-trivial monoid (idempotents, Cartan
matrix, quiver) translate straightforwardly into algorithms. Those
algorithms have been implemented by the authors, in the open source
mathematical system \texttt{Sage}~\cite{Sage}, in order to support
their own research. The code is publicly available from the
\texttt{Sage-Combinat} patch server~\cite{Sage-Combinat}, and is being
integrated into the main \texttt{Sage} library and generalized to
larger classes of monoids in collaboration with other
\texttt{Sage-Combinat} developers. It is also possible to delegate all
the low-level monoid calculations (Cayley graphs, $\JJ$-order, ...) to
the blazingly fast $C$ library \texttt{Semigroupe} by Jean-Éric
Pin~\cite{Semigroupe}.

We start with a quick overview of the complexity of the algorithms.
\begin{proposition}
  In the statements below, $M$ is a $\JJ$-trivial monoid of
  cardinality $n$, constructed from a set of $m\leq n$ generators
  $s_1,\dots,s_m$ in some ambient monoid. The product in the ambient
  monoid is assumed to be $O(1)$. All complexity statements are upper
  bounds, with no claim for optimality. In practice, the number of
  generators is usually small; however the number of idempotents,
  which condition the size of the Cartan matrix and of the quiver, can
  be as large as $2^m$.
  \begin{enumerate}[(a)]
  \item \label{complexity.cayley} Construction of the left / right
    Cayley graph: $O(nm)$ (in practice it usually requires little more
    than $n$ operations in the ambient monoid);
  \item \label{complexity.sorting} Sorting of elements according to
    $\JJ$-order: $O(nm)$;
  \item \label{complexity.idempotents} Selection of idempotents:
    $O(n)$;
  \item \label{complexity.symbols} Calculation of all left and right
    symbols: $O(nm)$;
  \item \label{complexity.cartan_matrix} Calculation of the Cartan
    matrix: $O(nm)$;
  \item \label{complexity.quiver} Calculation of the quiver: $O(n^2)$.
  \end{enumerate}
\end{proposition}
\begin{proof}
  \ref{complexity.cayley}: See~\cite{Froidure_Pin.1997}

  \ref{complexity.sorting}: This is a topological sort calculation for
  the two sided Cayley graph which has $n$ nodes and $2nm$ edges.

  \ref{complexity.idempotents}: Brute force selection.

  For each of the following steps, we propose a simple algorithm
  satisfying the claimed complexity.

  \ref{complexity.symbols}: Construct, for each element $x$ of the
  monoid, two bit-vectors $l(x)=(l_1,\dots,l_m)$ and
  $r(x)=(r_1,\dots,r_m)$ with $l_i=\delta_{s_ix,x}$ and
  $r_i=\delta_{xs_i,x}$. This information is trivial to extract in
  $O(nm)$ from the left and right Cayley graphs, and could typically
  be constructed as a side effect of \ref{complexity.cayley}. Those
  bit-vectors describe uniquely $\laut{x}$ and $\raut{x}$. From that,
  one can recover all $\lfix{x}$ and $\rfix{x}$ in $O(nm)$: as a
  precomputation, run through all idempotents $e$ of $M$ to construct
  a binary prefix tree $T$ which maps $l(e)=r(e)$ to $e$; then, for
  each $x$ in $M$, use $T$ to recover $\lfix{x}$ and $\rfix{x}$ from
  the bit vectors $l(x)$ and $r(x)$.

  \ref{complexity.cartan_matrix}: Obviously $O(n)$ once all left and
  right symbols have been calculated; so $O(nm)$ altogether.

  \ref{complexity.quiver}: A crude algorithm is to compute all
  products $xy$ in the monoid, check whether the factorization is
  compatible, and if yes cross the result out of the quiver (brute
  force sieve). This can be improved by running only through the
  products $xy$ with $\rfix{x}=\lfix{y}$; however this does not change
  the worst case complexity (consider a monoid with only $2$
  idempotents $0$ and $1$, like $\N^m$ truncated by any ideal
  containing all but $n-2$ elements, so that $\lfix{x}=\rfix{x}=1$ for
  all $x\ne 0$).
\end{proof}

We conclude with a sample session illustrating typical calculations,
using \texttt{Sage 4.5.2} together with the \texttt{Sage-Combinat}
patches, running on \texttt{Ubuntu Linux 10.5} on a \texttt{Macbook
  Pro 4.1}. Note that the interface is subject to minor changes before
the final integration into Sage. The authors will gladly provide help
in using the software.

We start by constructing the $0$-Hecke monoid of the symmetric group
$W=\sg[4]$, through its action on $W$:
\begin{sageexample}
  sage: W = SymmetricGroup(4)
  sage: S = semigroupe.AutomaticSemigroup(W.simple_projections(), W.one(),
  ...                  by_action = True, category=FiniteJTrivialMonoids())
  sage: S.cardinality()
  24
\end{sageexample}
We check that it is indeed $\JJ$-trivial, and compute its $8$ idempotents:
\begin{sageexample}
  sage: S._test_j_trivial()
  sage: S.idempotents()
  [[], [1], [2], [3], [1, 3], [1, 2, 1], [2, 3, 2], [1, 2, 1, 3, 2, 1]]
\end{sageexample}
Here is its Cartan matrix and its quiver:
\begin{sageexample}
    sage: S.cartan_matrix_as_graph().adjacency_matrix(), S.quiver().adjacency_matrix()
    (
    [0 0 0 0 0 0 0 0]  [0 0 0 0 0 0 0 0]
    [0 0 1 0 1 1 0 0]  [0 0 1 0 1 1 0 0]
    [0 1 0 0 1 0 0 0]  [0 1 0 0 0 0 0 0]
    [0 0 0 0 0 0 0 0]  [0 0 0 0 0 0 0 0]
    [0 1 1 0 0 0 0 0]  [0 1 0 0 0 0 0 0]
    [0 1 0 0 0 0 1 1]  [0 1 0 0 0 0 1 1]
    [0 0 0 0 0 1 0 1]  [0 0 0 0 0 1 0 0]
    [0 0 0 0 0 1 1 0], [0 0 0 0 0 1 0 0]
    )
\end{sageexample}

In the following example, we check that, for any of the
$318$ posets $P$ on $6$ vertices, the Cartan matrix $m$ of the monoid
$\OR(P)$ of order preserving nondecreasing functions on $P$ is
unitriangular. To this end, we check that the digraph having $m-1$ as
adjacency matrix is acyclic.
% Slightly less cluttered non parallel version:
% \begin{sageexample}
%   sage: from sage.combinat.j_trivial_monoids import *
%   sage: for P in Posets(4):
%   ...      assert DiGraph(NDPFMonoidPoset(P).cartan_matrix()-1).is_directed_acyclic()
% \end{sageexample}
\begin{sageexample}
sage: from sage.combinat.j_trivial_monoids import *
sage: @parallel
...def check_cartan_matrix(P):
...    return DiGraph(NDPFMonoidPoset(P).cartan_matrix()-1).is_directed_acyclic()
sage: time all(res[1] for res in check_cartan_matrix(list(Posets(6))))
CPU times: user 5.68 s, sys: 2.00 s, total: 7.68 s
Wall time: 255.53 s
True
\end{sageexample}
Note: the calculation was run in parallel on two processors, and the
displayed CPU time is just that of the master process, which is not
much relevant. The same calculation on a eight processors machine
takes about 71 seconds.
% On Massena:
% CPU times: user 4.92 s, sys: 1.78 s, total: 6.69 s
% Wall time: 71.23 s
% True

We conclude with the calculation of the representation theory of a
larger example (the monoid $\unitribool_n$ of unitriangular Boolean
matrices). The current implementation is far from optimized: in
principle, the cost of calculating the Cartan matrix should be of the
same order of magnitude as generating the monoid. Yet, this
implementation makes it possible to explore routinely, if not
instantly, large Cartan matrices or quivers that were completely out
of reach using general purpose representation theory software.
\begin{sageexample}
  M = semigroupe.UnitriangularBooleanMatrixSemigroup(6)
  Loading Sage library. Current Mercurial branch is: combinat
  sage: time M.cardinality()
  CPU times: user 0.14 s, sys: 0.02 s, total: 0.16 s
  Wall time: 0.16 s
  32768
  sage: time M.cartan_matrix()
  CPU times: user 27.50 s, sys: 0.09 s, total: 27.59 s
  Wall time: 27.77 s
  4824 x 4824 sparse matrix over Integer Ring
  sage: time M.quiver()
  CPU times: user 512.73 s, sys: 2.81 s, total: 515.54 s
  Wall time: 517.55 s
  Digraph on 4824 vertices
\end{sageexample}
Figure~\ref{fig:unitribool4} displays the results in the case $n=4$.

%%%%%%%%%%%%%%%%%%%%%%%%%%%%%%%%%%%%%%%%%%%%%%%%%%%%%%%%%%%%%%%%%%%%%%%%%%%%%% 
\section{Monoid of order preserving regressive functions on a poset $P$}
%%%%%%%%%%%%%%%%%%%%%%%%%%%%%%%%%%%%%%%%%%%%%%%%%%%%%%%%%%%%%%%%%%%%%%%%%%%%%%
\label{sec:NPDF}

In this section, we discuss the monoid $\OR(P)$ of order preserving
regressive functions on a poset $P$.  Recall that this is the monoid
of functions $f$ on $P$ such that for any $x\leq y \in P$, $x.f\leq x$ and 
$x.f\leq y.f$.

In Section~\ref{sub:combIdem}, we discuss constructions for idempotents in 
$\OR(P)$ in terms of the image sets of the idempotents, as well as methods for 
obtaining $\lfix f$ and $\rfix f$ for any given function $f$.  In Section~\ref{sub:cartNDPF}, 
we show that the Cartan matrix for $\OR(P)$ is upper uni-triangular with respect to the lexicographic order associated to any linear 
extension of $P$.  In Section~\ref{sub:ndpfSemis}, we specialize to $\OR(L)$ where $L$ is a 
meet semi-lattice, describing a minimal generating set of idempotents.  Finally, in 
Section~\ref{sub:ndpfOrthIdem}, we describe a simple construction for a set of orthogonal 
idempotents in $\NDPF_N$, and present a conjectural construction for orthogonal idempotents 
for $\OR(L)$.

\subsection{Combinatorics of idempotents}
\label{sub:combIdem}

The goal of this section is to describe the idempotents in $\OR(P)$ using
order considerations. We begin by giving the definition of joins, even in the setting when
the poset $P$ is not a lattice.

\begin{definition}
  Let $P$ be a finite poset and $S\subseteq P$. Then $z\in P$ is called \emph{a
    join} of $S$ if $x \leq z$ holds for any $x\in S$, and $z$ is minimal with
  that property.

  We denote $\joins(S)$ the set of joins of $S$, and $\joins(x,y)$ for short
  if $S=\{x,y\}$. If $\joins(S)$ (resp. $\joins(x,y)$) is a singleton (for
  example because $P$ is a lattice) then we denote $\bigvee S$ (resp. $x\vee y$)
  the unique join.  Finally, we define $\joins(\emptyset)$ to be the set of minimal elements in $P$.
\end{definition}

\begin{lemma}
  \label{lemma.fix}
  Let $P$ be some poset, and $f\in \OR(P)$. If $x$ and $y$ are fixed
  points of $f$, and $z$ is a join of $x$ and $y$, then $z$ is a fixed
  point of $f$.
\end{lemma}
\begin{proof}
  Since $x\leq z$ and $y\leq z$, one has $x=x.f \leq z.f$ and $y=y.f \leq z.f$.
  Since furthermore $z.f\le z$, by minimality of $z$ the equality $z.f = z$ must hold.
\end{proof}

\begin{lemma}
  \label{lemma.sup}
  Let $I$ be a subset of $P$ which contains all the minimal elements
  of $P$ and is stable under joins. Then, for any $x\in P$, the set
  $\{y\in I\suchthat y\leq x\}$ admits a unique maximal element which
  we denote by $\sup_I(x)\in I$. Furthermore, the map $\sup_I:
  x\mapsto \sup_I(x)$ is an idempotent in $\OR(P)$.
\end{lemma}
\begin{proof}
   For the first statement, suppose for some $x \not \in I$ there are two 
   maximal elements $y_1$ and $y_2$ in $\{y\in I\suchthat y\leq x\}$.  
   Then the join $y_1 \wedge y_2 < x$, since otherwise $x$ would be a join of $y_1$ 
   and $y_2$, and thus $x\in I$ since $I$ is join-closed.  But this contradicts the maximality 
   of $y_1$ and $y_2$, so the first statement holds.

  Using that $\sup_I(x)\leq x$ and $\sup_I(x)\in I$, $e:=\sup_I$ is
  a regressive idempotent by construction. Furthermore, it is is order
  preserving: for $x\leq z$, $x.e$ and $z.e$ must be comparable or
  else there would be two maximal elements in $I$ under $z$.  Since
  $z.e$ is maximal under $z$, we have $z.e\geq x.e$.
\end{proof}

Reciprocally, all idempotents are of this form:
\begin{lemma}
  \label{lemma.idem_image}
  Let $P$ be some poset, and $f\in \OR(P)$ be an idempotent. Then the
  image $\im(f)$ of $f$ satisfies the following:
  \begin{enumerate}
  \item \label{item.min} All minimal elements of $P$ are contained in $\im(f)$.
  \item \label{item.fix} Each $x \in \im(f)$ is a fixed point of $f$.
  \item \label{item.join} The set $\im(f)$ is stable under joins: if
    $S\subseteq \im(f)$ then $\joins(S)\subseteq \im(f)$ .
  \item \label{item.image} For any $x\in P$, the image $x.f$ is the upper bound $\sup_{\im(f)}(x)$.
  \end{enumerate}
\end{lemma}
\begin{proof}  
  Statement~\eqref{item.min} follows from the fact that $x.f\le x$ so that minimal elements
  must be fixed points and hence in $\im(f)$.

  For any $x=a.f$, if $x$ is not a fixed point then $x.f=(a.f).f\neq a.f$, contradicting the 
  idempotence of $f$.  Thus, the second statement holds.
  
  Statement~\eqref{item.join} follows directly from the second statement and Lemma~\ref{lemma.fix}.
  
  If $y\in \im(f)$ and $y\leq x$ then $y = y.f \leq x.f$.  Since this holds for every element of 
  $\{y \in \im(f) \mid y\leq x \}$ and $x.f$ is itself in this set, statement~\eqref{item.image} holds.
\end{proof}

Thus, putting together Lemmas~\ref{lemma.sup}
and~\ref{lemma.idem_image} one obtains a complete description of the
idempotents of $\OR(P)$.
\begin{proposition}
  \label{proposition.idempotents.OO}
  The idempotents of $\OR(P)$ are given by the maps $\sup_I$, where $I$
  ranges through the subsets of $P$ which contain the minimal elements
  and are stable under joins.
\end{proposition}

For $f\in \OR(P)$ and $y\in P$, let $f^{-1}(y)$ be the fiber of $y$ under $f$, that is, the set of
all $x\in P$ such that $x.f=y$.
\begin{definition}
Given $S$ a subset of a finite poset $P$, set $C_0(S)=S$ and $C_{i+1}(S)=C_i(S) \cup 
\{x\in P \mid x \text{ is a join of some elements in } C_i(S) \}$.  Since $P$ is finite, there exists some $N$ 
such that $C_N(S)=C_{N+1}(S)$.  The \emph{join closure} is defined as this stable set, 
and denoted $C(S)$.  A set is \emph{join-closed} if $C(S)=S$. Define
\[
F(f) := \bigcup_{y\in P} \{ x\in f^{-1}(y) \mid x \text{ minimal in } f^{-1}(y) \}
\]
to be the collection of minimal points in the fibers of $f$.
\end{definition}

\begin{corollary}
Let $X$ be the join-closure of the set of minimal points of $P$.  Then $X$ is fixed by every 
$f\in \OR(P)$.
\end{corollary}

\begin{lemma}[Description of left and right symbols]
For any $f\in \OR(P)$, there exists a minimal idempotent $f_r$ whose image set is $C(\im(f))$, 
and $f_r=\rfix{f}$.  There also exists a minimal idempotent $f_l$ whose image set is $C(F(f))$, 
and $f_l = \lfix{f}$.
\end{lemma}

\begin{proof}
The $\rfix{f}$ must fix every element of $\im(f)$, and the image of $\rfix{f}$ must be join-closed
by Lemma~\ref{lemma.idem_image}.
$f_r$ is the smallest idempotent satisfying these requirements, and is thus the $\rfix{f}$.

Likewise, $\lfix{f}$ must fix the minimal elements of each fiber of $f$, and so must fix all of $C(F(f))$.  
For any $y \not \in F(f)$, find $x\leq y$ such that $x.f=y.f$ and $x\in F(f)$.  Then 
$x= x.f_l \leq y.f_l \leq y$.
For any $z$ with $x\leq z \leq y$, we have $x.f\leq z.f \leq y.f=x.f$, so $z$ is in the 
same fiber as $y$.  Then we have $(y.f_l).f =y.f$, so $f_l$ fixes $f$ on the left.  
Minimality then ensures that $f_l=\lfix{f}$.
\end{proof}

Let $P$ be a poset, and $P'$ be the poset obtained by removing a
maximal element $x$ of $P$. Then, the following rule holds:
\begin{proposition}[Branching of idempotents]
  Let $e=\sup_I$ be an idempotent in $\OR(P')$. If $I\subseteq P$ is
  still stable under joins in $P$, then there exist two idempotents in
  $\OR(P)$ with respective image sets $I$ and $I\cup \{x\}$.
  Otherwise, there exists an idempotent in $\OR(P)$ with image set
  $I\cup \{x\}$. Every idempotent in $\OR(P)$ is uniquely obtained by
  this branching.
\end{proposition}
\begin{proof}
  This follows from straightforward reasoning on the subsets $I$ which contain the
  minimal elements and are stable under joins, in $P$ and in $P'$.
\end{proof}

\subsection{The Cartan matrix for $\OR(P)$ is upper uni-triangular}
\label{sub:cartNDPF}

We have seen that the left and right fix of an element of $\OR(P)$
can be identified with the subsets of $P$ closed under joins. We put a
total order $\leq_\lex$ on such subsets by writing them as bit vectors
along a linear extension $p_1,\dots,p_n$ of $P$, and comparing those
bit vectors lexicographically.

\begin{proposition}
  Let $f\in \OR(P)$. Then, $\im(\lfix f) \leq_\lex \im(\rfix f)$, with
  equality if and only if $f$ is an idempotent.
\end{proposition}
\begin{proof}
  Let $n=|P|$ and $p_1,\ldots,p_n$ a linear extension of $P$.
  For $k\in \{0,\dots,n\}$ set respectively $L_k=\im(\lfix f) \cap
  \{p_1,\dots,p_k\}$ and $R_k = \im(\rfix f) \cap \{p_1,\dots,p_k\}$.

  As a first step, we prove the property $(H_k)$: if $L_k=R_k$ then
  $f$ restricted to $\{p_1,\dots,p_k\}$ is an idempotent with image
  set $R_k$. Obviously, $(H_0)$ holds. Take now $k>0$ such that
  $L_k=R_k$; then $L_{k-1}=R_{k-1}$ and we may use by induction
  $(H_{k-1})$.

  Case 1: $p_k\in F(f)$, and is thus the smallest point in its fiber.  This implies that
  $p_k\in L_k$, and by assumption, $L_k=R_k$. 
  By $(H_{k-1})$, $p_k.f<_\lex p_k$ gives a contradiction:
  $p_k.f\in R_{k-1}$, and therefore $p_k.f$ is in the same fiber as
  $p_k$. Hence $p_k.f=p_k$.

  Case 2: $p_k \in C(F(f))=\im(\lfix{f})$, but $p_k \not \in F(f)$.   Then $p_k$ is a join of 
  two smaller elements $x$ and $y$ of $L_k=R_k$; in particular, $p_k\in R_k$. By induction, 
  $x$ and $y$ are fixed by $f$, and therefore $p_k.f=p_k$ by Lemma~\ref{lemma.fix}.

  Case 3: $p_k\not \in C(F(f))=\im(\lfix f)$; then $p_k$ is not a minimal element
  in its fiber; taking $p_i<_\lex p_k$ in the same fiber, we have $(p_k.f).f =
  (p_i.f).f = p_i.f = p_k.f$. Furthermore, $R_k=R_{k-1} =
  \{p_1,\dots,p_{k-1}\}.f = \{p_1,\dots,p_k\}.f$.

  In all three cases above, we deduce that $f$ restricted to
  $\{p_1,\dots,p_k\}$ is an idempotent with image set $R_k$, as desired.

  If $L_n=R_n$, we are done. Otherwise, take $k$ minimal such that
  $L_k\ne R_k$. Assume that $p_k\in L_k$ but not in $R_k$. In
  particular, $p_k$ is not a join of two elements $x$ and $y$ in
  $L_{k-1}=R_{k-1}$; hence $p_k$ is minimal in its fiber, and by the
  same argument as in Case 3 above, we get a contradiction.
\end{proof}

\begin{corollary}
  The Cartan matrix of $\OR(P)$ is upper uni-triangular with respect to the
  lexicographic order associated to any linear extension of $P$.
\end{corollary}

\begin{problem}
  Find larger classes of monoids where this property still holds.
  Note that this fails for the $0$-Hecke monoid which is a submonoid
  of an $\OR(B)$ where $B$ is Bruhat order.
\end{problem}

\subsection{Restriction to meet semi-lattices}
\label{sub:ndpfSemis}

For the remainder of this section, let $L$ be a \emph{meet semi-lattice} and we consider 
the monoid $\OR(L)$. Recall that $L$ is a meet semi-lattice if every pair of elements $x,y\in L$
has a unique meet.

For $a\geq b$, define an idempotent $e_{a,b}$ in $\OR(L)$ by:
\[
    x.e_{a,b} =
    \begin{cases}
        x\wedge b & \text{if $x\leq a$,}\\
        x         & \text{otherwise.}
    \end{cases}
\]

\begin{remark}
  \label{remark.oo.eab}
  The function $e_{a,b}$ is the (pointwise) largest element of $\OR(L)$ such that $a.f=b$.

  For $a\geq b\geq c$, $e_{a,b} e_{b,c} = e_{a,c}$. In the case where
  $L$ is a chain, that is $\OR(L)=\NDPF_{|L|}$, those idempotents
  further satisfy the following braid-like relation: $e_{b,c} e_{a,b}
  e_{b,c} = e_{a,b} e_{b,c} e_{a,b} = e_{a,c}$.
\end{remark}
\begin{proof}
  The first statement is clear. Take now $a\geq b\geq c$ in a
  meet semi-lattice. For any $x\leq a$, we have $x.e_{a,b}=x\wedge b
  \leq b,$ so $x.(e_{a,b}e_{b,c}) = x\wedge b \wedge c = x \wedge c$,
  since $b\ge c$.  On the other hand, $x.e_{a,c} = x\wedge c$, which
  proves the desired equality.

  Now consider the braid-like relation in $\NDPF_{|L|}$.  Using the previous
  result, one gets that $e_{b,c} e_{a,b} e_{b,c}=e_{b,c} e_{a,c}$ and $e_{a,b}
  e_{b,c} e_{a,b}=e_{a,c} e_{a,b}$.  For $x> a$, $x$ is fixed by
  $e_{a,c}$, $e_{a,b}$ and $e_{b,c}$, and is thus fixed by the
  composition. The other cases can be checked analogously.
\end{proof}

\begin{proposition}
  The family $(e_{a,b})_{a,b}$, where $(a,b)$ runs through the covers of $L$, minimally generates the idempotents of $\OR(L)$.
\end{proposition}

\begin{proof}
Given $f$ idempotent in $\OR(L)$, we can factorize $f$ as a product of the idempotents $e_{a,b}$.  Take a 
linear extension of $L$, and recursively assume that $f$ is the identity on all 
elements above some least element $a$ of the linear extension.  
Then define a function $g$ by: 
\[
x.g= \begin{cases}
	a & \text{if $x=a$,} \\
	x.f & \text{otherwise.}
	\end{cases}
\]
We claim that $f = ge_{a,a.f},$ and $g \in \OR(L)$.  There are a number of cases that must be checked:
\begin{itemize}
\item Suppose $x<a$.  Then $x.ge_{a,a.f} = (x.f).e_{a,a.f} = x.f \wedge a.f = x.f$, since $x<a$ implies $x.f<a.f$.
\item Suppose $x>a$.  Then $x.ge_{a,a.f} = (x.f).e_{a,a.f} = x.e_{a,a.f} = x=x.f$, since $x$ is fixed by $f$ by assumption.
\item Suppose $x$ not related to $a$, and $x.f\leq a.f$.  Then $x.ge_{a,a.f} = (x.f).e_{a,a.f} = x.f$.
\item Suppose $x$ not related to $a$, and $a.f\leq x.f\leq a$.  By the idempotence of $f$ we have $a.f=a.f.f\le x.f.f\le a.f$, 
so $x.f=a.f$, which reduces to the previous case.
\item Suppose $x$ not related to $a$, but $x.f\leq a$.  Then by idempotence of $f$ we have $x.f=x.f.f\leq a.f$, 
reducing to a previous case.
\item For $x$ not related to $a$, and $x.f$ not related to $a$ or $x.f>a$, we have $x.f$ fixed by $e_{a,a.f}$, which 
implies that $x.ge_{a,a.f}=x.f$.
\item Finally for $x=a$ we have $a.ge_{a,a.f} = a.e_{a,a.f}=a\wedge a.f=a.f$.
\end{itemize}
Thus, $f = ge_{a,a.f}$.

For all $x\le a$, we have $x.f\le a.f\le a$, so that $x.g\le a.g=a$.  For all
$x>a$, we have $x$ fixed by $g$ by assumption, and for all other $x$, the
$\OR(L)$ conditions are inherited from $f$.  Thus $g$ is in $\OR(L)$.

For all $x\neq a$, we have $x.g=x.f=x.f.f$.  Since all $x>a$ are fixed by $f$, there is no $y$ such that $y.f=a$.  Then $x.f.f=x.g.g$ for all $x\neq a$.  Finally, $a$ is fixed by $g$, so $a=a.g.g$.  Thus $g$ is idempotent.

Applying this procedure recursively gives a factorization of $f$ into a composition of functions $e_{a,a.f}$.  We can further refine this factorization using Remark~\ref{remark.oo.eab} on each $e_{a,a.f}$ by
$e_{a,a.f}=e_{a_0,a_1}e_{a_1,a_2}\cdots e_{a_{k-1},a_k}$, where $a_0=a$, $a_k=a.f$, and
$a_i$ covers $a_{i-1}$ for each $i$.  Then we can express $f$ as a product of functions $e_{a,b}$ where $a$ covers $b$.

This set of generators is minimal because $e_{a,b}$ where $a$ covers $b$ is the pointwise largest function in $\OR(L)$ mapping $a$ to $b$.
\end{proof}

As a byproduct of the proof, we obtain a canonical factorization of any idempotent $f\in \OR(L)$.

\begin{example}
The set of functions $e_{a,b}$ do not in general generate $\OR(L)$.  Let $L$ be the Boolean lattice on three elements.  Label the nodes of $L$ by triples $ijk$ with $i,j,k\in \{0,1\}$, and $abc\geq ijk$ if $a\leq i, b\leq j, c\leq k$.

Define $f$ by $f(000)=000$, $f(100)=110, f(010)=011, f(001)=101$, and $f(x)=111$ for all other $x$.  Simple inspection shows that $f\neq ge_{a,a.f}$ for any choice of $g$ and $a$.  
\end{example}

\subsection{Orthogonal idempotents}
\label{sub:ndpfOrthIdem}

For $\{1,2,\ldots, N\}$ a chain, one can explicitly write down
orthogonal idempotents for $\NDPF_N$. Recall that the minimal
generators for $\NDPF_N$ are the elements $\pi_i =e_{i+1,i}$ and that
$\NDPF_N$ is the quotient of $H_0(\sg[n])$ by the extra relation
$\pi_i\pi_{i+1}\pi_i = \pi_{i+1}\pi_i$, via the quotient map
$\pi_i\mapsto \pi_i$. By analogy with the $0$-Hecke algebra, set
$\pi_i^+=\pi_i$ and $\pi_i^-=1-\pi_i$.

We observe the following relations, which can be checked easily.
\begin{lemma}\label{lem:ndpfrels}
Let $k=i-1$.  Then the following relations hold:
    \begin{enumerate}
      \item $\pi_{i-1}^+ \pi_i^+ \pi_{i-1}^+ = \pi_i^+ \pi_{i-1}^+$,
      \item $\pi_{i-1}^- \pi_i^- \pi_{i-1}^- = \pi_{i-1}^- \pi_i^-$,
      \item $\pi_i^+ \pi_{i-1}^- \pi_i^+ = \pi_i^+ \pi_{i-1}^-$,
      \item $\pi_i^- \pi_{i-1}^+ \pi_i^- = \pi_{i-1}^+ \pi_i^-$,
      \item $\pi_{i-1}^+ \pi_i^- \pi_{i-1}^+ = \pi_i^- \pi_{i-1}^+$,
      \item $\pi_{i-1}^- \pi_i^+ \pi_{i-1}^- = \pi_{i-1}^- \pi_i^+$.
    \end{enumerate}
\end{lemma}

\begin{definition}
  Let $D$ be a \emph{signed diagram}, that is an assignment of a $+$
  or $-$ to each of the generators of $\NDPF_N$. By abuse of notation,
  we will write $i\in D$ if the generator $\pi_i$ is assigned a $+$
  sign.  Let $P=\{ P_1, P_2, \ldots, P_k\}$ be the partition of the
  generators such that adjacent generators with the same sign are in
  the same set, and generators with different signs are in different
  sets.  Set $\epsilon(P_i)\in \{+,-\}$ to be the sign of the subset $P_i$.
  Let $\longest_{P_i}^{\epsilon(P_i)}$ be the longest element in the generators in
  $P_i$, according to the sign in $D$.  Define:
\begin{itemize}
  \item $L_D := \longest_{P_1}^{\epsilon(P_1)}\,\longest_{P_2}^{\epsilon(P_2)}\,\cdots \,\longest_{P_k}^{\epsilon(P_k)}$, 
  \item $R_D := \longest_{P_k}^{\epsilon(P_k)}\,\longest_{P_k-1}^{\epsilon(P_{k-1})}\,\cdots \,\longest_{P_1}^{\epsilon(P_1)}$, 
  \item and $C_D := L_DR_D$.
\end{itemize}
\end{definition}

\begin{example}
\label{example.D}
Let $D=++++---++$.  Then $P=\{ \{1,2,3,4\}, \{5,6,7\}, \{8,9\}  \}$, and the associated 
long elements are: $\longest_{P_1}^+=\pi_4^+ \pi_3^+ \pi_2^+ \pi_1^+$, 
$\longest_{P_2}^-=\pi_5^- \pi_6^- \pi_7^-$, and $\longest_{P_3}^+=\pi_9^+ \pi_8^+$. Then
\begin{equation*}
\begin{split}
	L_D &= \longest_{P_1}^+ \longest_{P_2}^- \longest_{P_3}^+ 
	= (\pi_4^+ \pi_3^+ \pi_2^+ \pi_1^+) (\pi_5^- \pi_6^- \pi_7^-) (\pi_9^+ \pi_8^+),\\
	R_D &= \longest_{P_3}^+ \longest_{P_2}^- \longest_{P_1}^+
	= (\pi_9^+ \pi_8^+) (\pi_5^- \pi_6^- \pi_7^-) (\pi_4^+ \pi_3^+ \pi_2^+ \pi_1^+).
\end{split}
\end{equation*}
\end{example}

The elements $C_D$ are the images, under the natural quotient map from
the $0$-Hecke algebra, of the \emph{diagram demipotents} constructed
in~\cite{Denton.2010.FPSAC,Denton.2010}. An element $x$ of an algebra
is \emph{demipotent} if there exists some finite integer $n$ such that
$x^n=x^{n+1}$ is idempotent. It was shown
in~\cite{Denton.2010.FPSAC,Denton.2010} that, in the $0$-Hecke
algebra, raising the diagram demipotents to the power $N$ yields a set
of primitive orthogonal idempotents for the $0$-Hecke algebra. It
turns out that, under the quotient to $\NDPF_N$, these elements $C_D$
are right away orthogonal idempotents, which we prove now.

\begin{remark}
  \label{remark.fix_i}
  Fix $i$, and assume that $f$ is an element in the monoid generated
  by $\pi^-_{i+1},...,\pi^-_N$ and $\pi^+_{i+1},...,\pi^+_N$. Then,
  applying repeatedly Lemma~\ref{lem:ndpfrels} yields
  $$\pi^-_i f \pi^-_i = \pi^-_i f \qquad \text{and} \qquad \pi^+_i f \pi^+_i = f \pi^+_i\,.$$
\end{remark}

The following proposition states that the elements $C_D$ are also the
images of Norton's generators of the projective modules of the
$0$-Hecke algebra through the natural quotient map to $\NDPF_N$.
\begin{proposition}
  \label{proposition.norton_ndpf}
  Let $D$ be a signed diagram. Then,
  \begin{displaymath}
    C_D = \prod_{i=1,\dots,n,\  i \not\in D} \pi^-_i  \quad \prod_{i=n,\dots,1,\  i\in D} \pi^+_i\,.
  \end{displaymath}
  In other words $C_D$ reduces to one of the following two forms:
  \begin{itemize}
  \item $C_D = (\longest_{P_1}^-\longest_{P_3}^-\cdots \longest_{P_{2k\pm 1}}^-) (\longest_{P_2}^+\longest_{P_4}^+\cdots \longest_{P_{2k}}^+)$, or
  \item $C_D = (\longest_{P_2}^-\longest_{P_4}^-\cdots \longest_{P_{2k}}^-) (\longest_{P_1}^+\longest_{P_3}^+\cdots \longest_{P_{2k\pm 1}}^+)$.
  \end{itemize}
\end{proposition}
\begin{proof}
  Let $D$ be a signed diagram. If it is of the form $-E$, where $E$ is
  a signed diagram for the generators $\pi_2,\dots,\pi_{N-1}$, then
  using Remark~\ref{remark.fix_i},
  $$C_D = \pi^-_1 C_E \pi^-_1 = \pi^-_1 C_E\,.$$
  Similarly, if it is of the form $+E$, then:
  $$C_D = \pi^+_1 C_E \pi^+_1 = C_E \pi^+_1\,.$$
  Using induction on the isomorphic copy of $\NDPF_{N-1}$ generated by
  $\pi_2,\dots,\pi_{N-1}$ yields the desired formula.
\end{proof}

\begin{proposition}
  \label{proposition.idempotents.ndpf}
  The collection of all $C_D$ forms a complete set of orthogonal
  idempotents for $\NDPF_N$.
\end{proposition}
\begin{proof}
  First note that $C_D$ is never zero; for example, it is clear from
  Proposition~\ref{proposition.norton_ndpf} that the full expansion of
  $C_D$ has coefficient $1$ on $\prod_{i=n,\dots,1,\ i\in D} \pi^+_i$.

  Take now $D$ and $D'$ two signed diagrams. If they differ in the
  first position, it is clear that $C_D C_{D'} =0$. Otherwise, write
  $D = \epsilon E$, and $D' = \epsilon E'$. Then, using
  Remark~\ref{remark.fix_i} and induction,
  \begin{equation*}
    \begin{split}
      C_D C_D' & = \pi^\epsilon_1 C_E \pi^\epsilon_1 \pi^\epsilon_1 C_{E'} \pi^\epsilon_1
      = \pi^\epsilon_1 C_E \pi^\epsilon_1 C_{E'} \pi^\epsilon_1\\
      &= \pi^\epsilon_1 C_E C_{E'} \pi^\epsilon_1
      = \pi^\epsilon_1 \delta_{E,E'} C_E \pi^\epsilon_1
      = \delta_{D,D'} C_D\,.
    \end{split}
  \end{equation*}
  Therefore, the $C_D$'s form a collection of $2^{N-1}$ nonzero
  orthogonal idempotents, which has to be complete by cardinality.
\end{proof}

One can interpret the diagram demipotents for $\NDPF_N$ as branching
from the diagram demipotents for $\NDPF_{N-1}$ in the following way.
For any $C_D=L_DR_D$ in $\NDPF_{N-1}$, the leading term of $C_D$ will
be the longest element in the generators marked by plusses in $D$.
This leading idempotent has an image set which we will denote $\im(D)$
by abuse of notation. Now in $\NDPF_N$ we can associated two
`children' to $C_D$:
\[
C_{D+}=L_D\pi_{N}^+R_D \text{ and } C_{D-}=L_D\pi_{N}^-R_D.
\]
Then we have $C_{D+}+C_{D-}=C_D$, $\im(D+)=\im(D)$ and $\im(D-)=\im(D)\bigcup \{N\}$.
\bigskip

We now generalize this branching construction to any meet semi-lattice
to derive a conjectural recursive formula for a decomposition of the
identity into orthogonal idempotents. This construction relies on the
branching rule for the idempotents of $\OR(L)$, and the existence of
the maximal idempotents $e_{a,b}$ of Remark~\ref{remark.oo.eab}.
\medskip

Let $L$ be a meet semi-lattice, and fix a linear extension of $L$. For
simplicity, we assume that the elements of $L$ are labelled
$1,\dots,N$ along this linear extension. Recall that, by
Proposition~\ref{proposition.idempotents.OO}, the idempotents are
indexed by the subsets of $L$ which contain the minimal elements of
$L$ and are stable under joins. In order to distinguish subsets of
$\{1,\dots,N\}$ and subsets of, say, $\{1,\dots,N-1\}$, even if they
have the same elements, it is convenient to identify them with $+-$
diagrams as we did for $\NDPF_N$. The \emph{valid diagrams} are those
corresponding to subsets which contain the minimal elements and are
stable under joins. A prefix of length $k$ of a valid diagram is still
a valid diagram (for $L$ restricted to $\{1,\dots,k\}$), and they are
therefore naturally organized in a binary prefix tree.

Let $D$ be a valid diagram, $e=\sup_D$ be the corresponding
idempotent. If $L$ is empty, $D=\{\}$, and we set
$L_{\{\}}=R_{\{\}}=1$.  Otherwise, let $L'$ be the meet semi-lattice
obtained by restriction of $L$ to $\{1,\dots,N-1\}$, and $D'$ the
restriction of $D$ to $\{1,\dots,N-1\}$.

\begin{itemize}
\item[Case 1] $N$ is the join of two elements of $\im(D')$ (and in particular,
  $N\in \im(D)$). Then, set $L_D=L_{D'}$ and $R_D=R_{D'}$.
\item[Case 2] $N\in \im(D)$. Then, set $L_D=L_{D'} \pi_{N, N.e}$ and
  $R_D=\pi_{N, N.e}R_{D'}$.
\item[Case 3] $N\not \in \im(D)$. Then, set $L_D=L_{D'} (1-\pi_{N, N.e})$
  and $R_D=(1-\pi_{N, N.e})R_{D'}$.
\end{itemize}
Finally, set $C_D=L_D R_D$.

\begin{remark}[Branching rule]
  Fix now $D'$ a valid diagram for $L'$. If $N$ is the join of two
  elements of $I'$, then $C_{D'}=C_{D'+}$. Otherwise $C_{D'}=C_{D'-} +
  C_{D'+}$.

  Hence, in the prefix tree of valid diagrams, the two sums of all
  $C_D$'s at depth $k$ and at depth $k+1$ respectively
  coincide. Branching recursively all the way down to the root of the
  prefix tree, it follows that the elements $C_D$ form a decomposition
  of the identity. Namely,
  \begin{equation*}
    1 = \sum_{D \text{ valid diagram}} C_D\,.
  \end{equation*}
\end{remark}

\begin{conjecture}
\label{conjecture.demi}
  Let $L$ be a meet semi-lattice. Then, the set $\{C_D\suchthat
  D \text{ valid diagram}\}$ forms a set of demipotent elements for
  $\OR(L)$ which, raised each to a sufficiently high power, yield a set
  of primitive orthogonal idempotents.
\end{conjecture}

This conjecture is supported by
Proposition~\ref{proposition.idempotents.ndpf}, as well as by computer
exploration on all $1377$ meet semi-lattices with at most $8$ elements
and on a set of meet semi-lattices of larger size which were considered
likely to be problematic by the authors. In all cases, the demipotents
were directly idempotents, which might suggest that
Conjecture~\ref{conjecture.demi} could be strengthened to state that
the collection $\{C_D\suchthat D \text{ valid diagram}\}$ forms
directly a set of primitive orthogonal idempotents for $\OR(L)$.

%%%%%%%%%%%%%%%%%%%%%%%%%%%%%%%%%%%%%%%%%%%%%%%%%%%%%%%%%%%%%%%%%%%%%%%%%%%%%%

\bibliographystyle{web-alpha}

\bibliography{main}

\end{document}